\documentclass[journal,twocolumn,twoside]{IEEEtran}

\IEEEoverridecommandlockouts                                 
\usepackage{amssymb,latexsym,color,amsmath,pifont,epsfig,graphicx}
\usepackage{setspace,cite}
\usepackage{algorithm}
\usepackage{algorithmic}
\usepackage{subfig}  
\usepackage{algorithm}
\usepackage{algorithmic}
\usepackage{soul}
\usepackage{graphics}
\usepackage{multirow}
\usepackage{bm}
\usepackage{rotating}
\usepackage{epsfig}
\usepackage{amssymb,latexsym,color,amsmath,pifont,epsfig,mathtools}
\usepackage{graphicx,ctable,booktabs}
\usepackage{indentfirst}
\usepackage{dsfont}
\usepackage[justification=centering]{caption}
\usepackage{url}
\usepackage{color}
\usepackage{hyperref}
\DeclareMathOperator*{\trace}{trace}

\DeclareMathOperator*{\argmin}{argmin}

\DeclareMathOperator*{\minimize}{minimize}
\DeclareMathOperator*{\maximize}{maximize}
\DeclareMathOperator*{\subject}{subject~to}

\DeclareMathOperator*{\otherwise}{otherwise}

\definecolor{orange}{rgb}{1,0.5,0}

\newcommand{\DefinedAs}[0]{\mathrel{\mathop:}=}
\newcommand{\bbR}{\mathbb{R}}

\newcommand{\one}{\mathds{1}}

\newcommand{\card}{\mathbf{card}}

\newcommand{\mre}{\mathrm{e}}

\newcommand{\tpsi}{\tilde{\psi}}
\newcommand{\ds}{\displaystyle}
\newcommand{\htwo}{{\cal H}_2}

\newcommand{\bE}{\mathbf{E}}
\newcommand{\prox}{\mathbf{prox}}
\newcommand{\sign}{\mathrm{sign}}

\newtheorem{theorem}{Theorem}

\newtheorem{proposition}[theorem]{Proposition}

\newtheorem{remark}{Remark}

\newcommand{\enma}[1]   {\ensuremath{#1}}

\newcommand{\non}{\nonumber}

\newcommand{\beq}{\begin{equation}}
\newcommand{\eeq}{\end{equation}}
\newcommand{\bseq}{\begin{subequations}}
\newcommand{\eseq}{\end{subequations}}
\newcommand{\beqn}{\begin{eqnarray}}
\newcommand{\eeqn}{\end{eqnarray}}
\newcommand{\ba}{\begin{array}}
\newcommand{\ea}{\end{array}}
\newcommand{\bct}{\begin{center}}
\newcommand{\ect}{\end{center}}
\newcommand{\btmz}{\begin{itemize}}
\newcommand{\etmz}{\end{itemize}}
\newcommand{\benum}{\begin{enumerate}}
\newcommand{\eenum}{\end{enumerate}}

\newcommand{\norm}[1]{\| #1 \|}                 

\newcommand{\diag}      {\enma{\mathrm{diag}}}

\newcommand{\inner}[2]{\left\langle #1,#2 \right\rangle}

\newcommand{\matbegin}{
        \left[
}
\newcommand{\matend}{
        \right]
}

\newcommand{\tbo}[2]{
  \matbegin \begin{array}{c}
       #1 \\ #2
       \end{array} \matend }

\newcommand{\be}{\begin{equation}}
\newcommand{\ee}{\end{equation}}

\newcommand{\cplxs}{ C\kern -.35em \rule{0.03 em}{.7 ex}~   }

\def\complex{\hbox{C\kern -.45em \rule{0.03 em}{1.5 ex}}~}

\newcommand{\bi}{\begin{itemize}}
\newcommand{\ei}{\end{itemize}}

\begin{document}

\title{\LARGE Topology {design for stochastically-forced} consensus networks}

\author{Sepideh Hassan-Moghaddam and Mihailo R.\ Jovanovi\'c
\thanks{Financial support from the $3$M Graduate Fellowship, the UMN Informatics Institute Transdisciplinary Faculty Fellowship, and the National Science Foundation under award ECCS-1407958 is gratefully acknowledged.}
\thanks{Sepideh Hassan-Moghaddam and Mihailo R.\ Jovanovi\'c are with the Department of Electrical and Computer Engineering, University of Minnesota, Minneapolis, MN 55455. E-mails: hassa247@umn.edu, mihailo@umn.edu.}
}

\maketitle

    \begin{abstract}
We study an optimal control problem aimed at achieving a desired tradeoff between the network coherence and communication requirements in the distributed controller. Our objective is to add a certain number of edges to an undirected network, with a known graph Laplacian, in order to optimally enhance closed-loop performance. To promote controller sparsity, we introduce $\ell_1$-regularization into the optimal ${\cal H}_2$ formulation and cast the design problem as a semidefinite program. We derive a Lagrange dual, {provide interpretation of dual variables}, and exploit structure of the optimality conditions for undirected networks to develop customized {proximal gradient and Newton} algorithms that are well-suited for large problems. We illustrate that our algorithms can solve the problems with more than million edges in the controller graph in a few minutes, on a PC. We also exploit structure of connected resistive networks to demonstrate how additional edges can be systematically added in order to minimize the ${\cal H}_2$ norm of the closed-loop system.
    \end{abstract}


	\vspace*{-2ex}
\section{Introduction}

	
Conventional optimal control of distributed systems relies on centralized implementation of control policies. In large networks of dynamical systems, centralized information processing imposes a heavy burden on individual nodes and is often infeasible. This motivates the development of distributed control strategies that require limited information exchange between the nodes to reach consensus or guarantee synchronization. Over the last decade, a vast body of literature has dealt with analysis, fundamental performance limitations, and design of distributed averaging protocols; e.g., see~\cite{mesege10,xiaboy04,kimmes06,xiaboykim07,barhes07,ghoboysab08,zelmes11,bamjovmitpat12}.


	
Optimal design of the edge weights for networks with pre-specified topology has received significant attention. In~\cite{xiaboy04}, the design of the fastest averaging protocol for undirected networks was cast as a semidefinite program (SDP). Two customized algorithms, based on primal barrier interior-point (IP) and subgradient methods, were developed and the advantages of optimal weight selection over commonly used heuristics were demonstrated. Similar SDP characterization, for networks with state-dependent graph Laplacians, was provided in~\cite{kimmes06}. The allocation of symmetric edge weights that minimize the mean-square deviation from average for networks with additive stochastic disturbances was solved in~\cite{xiaboykim07}. A related problem, aimed at minimizing the total effective resistance of resistive networks, was addressed in~\cite{ghoboysab08}. In~\cite{zelmes11}, the edge Laplacian was used to provide graph-theoretic characterization of the ${\cal H}_2$ and ${\cal H}_\infty$ symmetric agreement protocols.


Network coherence quantifies the ability of distributed estimation and control strategies to guard against exogenous disturbances~\cite{barhes07,bamjovmitpat12}. The coherence is determined by the sum of reciprocals of the non-zero eigenvalues of the graph Laplacian and its scaling properties cannot be predicted by algebraic connectivity of the network. In~\cite{bamjovmitpat12}, performance limitations of spatially-localized consensus protocols {on regular lattices} were examined. {It was shown that} {\em {the fundamental limitations for large-scale networks are} dictated by the network topology rather than by the optimal selection of the edge weights\/}. {Moreover, epidemic spread in networks is strongly influenced by their topology~\cite{wanroysab08,prezarenyjadpap13,rammar16}.} {Thus, optimal topology design} represents an important challenge. It is precisely this problem, {for undirected consensus networks}, \mbox{that we address in the paper.}

{More specifically}, we study an optimal control problem aimed at achieving a {\em desired tradeoff\/} between the {\em network performance\/} and {\em communication requirements\/} in the distributed controller. Our goal is to add a certain number of edges to a given undirected network in order to optimally enhance the closed-loop performance. One of our key contributions is the formulation of topology design as an optimal control problem that admits convex characterization and is amenable to the development of efficient optimization algorithms. In our formulation, the plant network can contain disconnected components and optimal topology of the controller network is an integral part of the design. In general, this problem {is NP-hard~\cite{siamot16} and it} amounts to an intractable combinatorial search. Several references have examined convex relaxations or greedy algorithms to {design} topology that optimizes algebraic connectivity~\cite{ghoboy06} or network coherence~\cite{linfarjovALL12,zelschall13,farlinjovTAC14sync,sumshalygdor15}.
		
We tap on recent developments regarding sparse representations in conjunction with regularization penalties on the level of communication in a distributed controller. This allows us to formulate convex optimization problems that exploit the underlying structure and are amenable to the development of efficient optimization algorithms. To avoid combinatorial complexity, we approach optimal topology design using a sparsity-promoting optimal control framework {introduced in}~\cite{farlinjovACC11,linfarjovTAC13admm}. Performance is captured by the ${\cal H}_2$ norm of the closed-loop network and $\ell_1$-regularization is introduced to promote controller sparsity. {While this problem is in general nonconvex~\cite{linfarjovTAC13admm}, for undirected networks we show that it admits a convex characterization} with a non-differentiable objective function {and a positive definite constraint}. This problem can be transformed into an SDP {and}, for small size networks, the optimal solution can be computed using standard IP method solvers, {e.g., SeDuMi~\cite{sedumi} and SDPT3~\cite{sdpt3}.}

To enable design of large networks, we pay particular attention to the computational aspects of the {edge-addition} problem. We derive a Lagrange dual of the optimal control problem, {provide interpretation of dual variables}, and develop efficient {proximal} algorithms. Furthermore, building on preliminary work~\cite{mogjovACC15}, we specialize our algorithms to the problem of growing connected resistive networks {described in}~\cite{ghoboy06,ghoboysab08}. In this, the plant graph is connected and inequality constraints amount to non-negativity of controller edge weights. This  allows us to simplify optimality conditions and further improve computational efficiency of our customized algorithms.

Proximal gradient algorithms{\cite{parboy13}} and their accelerated variants{\cite{becteb09}} have recently found use in distributed optimization, statistics, machine learning, image and signal processing. They can be interpreted as generalization of standard gradient projection to problems with non-smooth and extended real-value objective functions. When the proximal operator is easy to evaluate, these algorithms are simple yet extremely efficient.

	For networks that can contain disconnected components and non-positive edge weights, {we show that} the proximal gradient algorithm iteratively updates the controller graph Laplacian via convenient use of the soft-thresholding operator. This extends the Iterative Shrinkage Thresholding Algorithm (ISTA) to optimal topology design of undirected networks. In contrast to the $\ell_1$-regularized least-squares, however, the step-size has to be selected to guarantee positivity of the second smallest eigenvalue of the closed-loop graph Laplacian. We combine the Barzilai-Borwein (BB) step-size initialization with backtracking to achieve this goal and enhance the rate of convergence. The biggest computational challenge comes from evaluation of the objective function and its gradient. We exploit problem structure to speed up computations and save memory. Finally, for the problem of growing connected resistive networks, the proximal algorithm simplifies to gradient projection which additionally improves the efficiency.
		
We also develop a customized algorithm based on the proximal Newton method. In contrast to the proximal gradient, this method sequentially employs the second-order Taylor series approximation of the smooth part of the objective function;~{e.g., see~\cite{leesunsau14}}. We use cyclic coordinate descent over the set of active variables to efficiently compute the Newton direction by consecutive minimization with respect to individual coordinates. Similar approach has been recently utilized in a number of applications, including sparse inverse covariance estimation in graphical models~\cite{hsisusdhirav14}.


{Both} of our customized {proximal} algorithms significantly outperform {a primal-dual IP method developed in~\cite{mogjovACC15}. It is worth noting that the latter is significantly faster than the general-purpose solvers. While the customized IP algorithm of~\cite{mogjovACC15}} with a simple diagonal preconditioner can solve the problems with hundreds of thousands of edges in the controller graph in several hours, on a PC, the customized algorithms based on proximal gradient and Newton methods can solve the problems with millions of edges in several minutes. {Furthermore, they} are considerably faster than the greedy algorithm with efficient rank-one \mbox{updates {developed in}~\cite{sumshalygdor15}.}

	
Our presentation is organized as follows. In Section~\ref{sec.setup}, we formulate the problem of optimal topology design for undirected networks subject to additive stochastic disturbances. In Section~\ref{sec.dual}, we derive a Lagrange dual of the sparsity-promoting optimal control problem, provide interpretation of dual variables, and construct dual feasible variables from the primal ones. In Section~\ref{sec.IPgen}, we develop customized algorithms based on the proximal gradient and Newton methods. In Section~\ref{sec.resistive}, we achieve additional speedup by specializing our algorithms to the problem of growing connected resistive networks. In Section~\ref{sec.examples}, we use computational experiments to design optimal topology of a controller graph for benchmark problems {and demonstrate efficiency of our algorithms.} In Section~\ref{sec.conclusion}, we provide a brief overview of the paper.

	\vspace*{-2ex}
\section{Problem formulation}
	\label{sec.setup}

We consider undirected consensus networks with $n$ nodes
	\beq
	\dot{\psi}
	\; = \;
	-  L_p  \, \psi  \; + \;   u \; + \; d
	\label{eq.plant}
	\eeq
where {$d$ and $u$ are the disturbance and control inputs, $\psi$ is the state of the network, and $L_p$ is a symmetric $n\times n$ matrix that represents graph Laplacian of the open-loop system, i.e., plant. The goal is to improve performance of a consensus algorithm in the presence of stochastic disturbances by adding a certain number of edges (from a given set of candidate edges). We formulate this problem as a feedback design problem with
	\beq
	u
	\; = \;
	- L_x \, \psi
	\non
	\eeq	
where the symmetric feedback-gain matrix $L_x$ is required to have the Laplacian structure. This implies that each node in~\eqref{eq.plant} forms control action using a weighted sum of the differences between its own state and the states of other nodes and that information is processed in a symmetric fashion. Since a nonzero $ij$th element of $L_x$ corresponds to an edge between the nodes $i$ and $j$, the communication structure in the controller graph is determined by the sparsity pattern of the matrix $L_x$.}

Upon closing the loop we obtain
	\begin{subequations}
	\label{eq.ss}
	\beq
	\dot{\psi}
	\; = \;
	-  \left( L_p  \, +  \,  L_x\right)  \psi  \; + \;  d.
	\label{eq.ss1}
	\eeq
{For a given $L_p$}, our objective is to design the topology for $L_x$ and the corresponding edge weights $x$ in order to achieve the desired tradeoff between controller sparsity and network performance. The performance is quantified by the steady-state variance amplification of the stochastically-forced network, from the white-in-time input $d$ to the performance output $\zeta$,
	\beq
	{
	\zeta
	\, \DefinedAs \,
	\tbo{Q^{1/2}}{0} \psi
	\, + \,
	\tbo{0}{R^{1/2}} u
	\, = \,
	\tbo{Q^{1/2}}{- R^{1/2} L_x} \psi
	}
	\label{eq.ss2}
	\eeq
	\end{subequations}
which penalizes deviation from consensus and control effort. {Here, $Q = Q^T \succeq 0$ and $R = R^T \succ 0$ are the state and control weights in the standard quadratic performance index.}	

The interesting features of this problem come from structural restrictions on the {Lalpacian matrices $L_p$ and $L_x$}. {Both} of them are symmetric and are restricted to having an eigenvalue at zero with the corresponding eigenvector of all ones,
	\beq
    L_p \, \one
    \; = \;
    0,
    ~~~
    L_x \, \one
    \; = \;
    0.
    \label{eq.Lp-Lx}
    \eeq
{Since each node uses relative information exchange with its neighbors to update its state, in the presence of white noise, the average mode $\bar{\psi} (t)\DefinedAs(1/n) \, \one^T \psi (t)$ experiences a random walk and its variance increases linearly with time. To make the average mode unobservable from the performance output $\zeta$, the matrix $Q$ is also restricted to having an eigenvalue at zero associated with the vector of all ones, $Q \, \one = 0$. Furthermore,} to guarantee observability of the remaining eigenvalues of $L_p$, we consider state weights that are positive definite on the orthogonal complement of the subspace spanned by the vector of all ones,
	$
    Q
     +
    (1/n)
    \,
    \one \one^{T}
    \succ
     0;
    $
e.g.,
	$
	Q = I - (1/n)
    \,
    \one \one^{T}
    $
penalizes mean-square deviation from the network average.

In what follows, we express $L_x$ as 	
	\be
	L_x
	\;
	\DefinedAs
	\;
    	\ds{\sum_{l \, = \, 1}^{m}} \; x_l \, \xi_l \, \xi_l^T
    	\; = \;
	E \, \diag \left( x \right) E^T
	\label{eq.Lx}
	\ee
where $E$ is the incidence matrix of the controller graph $L_x$, $m$ is the number of edges in $L_x$, and $\diag \left( x \right)$ is a diagonal matrix containing the vector of the edge weights $x \in \bbR^m$. {The matrix $E$ is given and it determines the set of candidate edges. It is desired to select a subset of this set in order to balance the closed-loop performance with the number of added edges.} Vectors $\xi_l \in \mathbb{R}^{n}$ determine the columns of $E$ and they signify the connection with weight $x_l$ between nodes $i$ and $j$: the $i$th and $j$th entries of $\xi_l$ are $1$ and $-1$ and all other entries are equal to $0$. Thus, $L_x$ given by~\eqref{eq.Lx} satisfies structural requirements on the controller graph Laplacian in~\eqref{eq.Lp-Lx} by construction.



To achieve consensus in the absence of disturbances, the closed-loop network has to be connected~\cite{mesege10}. Equivalently, the second smallest eigenvalue of the closed-loop graph Laplacian, $L \DefinedAs L_p + L_x$, has to be positive, i.e.,  $L$ has to be positive definite on $\one^{\bot}$. This amounts to positive definiteness of the ``strengthened'' graph Laplacian of the closed-loop network
	\begin{subequations}
	\label{eq.G}
	\beq
	\ba{rrl}
	G
	&
	\!\!
	\DefinedAs
	\!\!
	&
	L_p
	\; + \;
	L_x
	\; + \;
	(1/n) \, \one \one^T
	\\[0.15cm]
	&
	\!\!
	=
	\!\!
	&
	G_p
	\; + \;
	E \, \diag \left( x \right) E^T
	\;
	\succ
	\;
	0
	\ea
	\label{eq.Ga}
   	\eeq	
where
	\beq
	G_p
	\; \DefinedAs \;
	L_p  \; + \; (1/n) \, \one \one^T.
	\label{eq.Gb}
	\eeq
Structural restrictions~\eqref{eq.Lp-Lx} on the Laplacian matrices introduce an additional constraint on {the matrix} $G$,
   	\beq
	G \, \one
	\; = \;
	\one.
	\label{eq.Gc}
	\eeq
	\end{subequations}

	\vspace*{-4ex}	
\subsection{Design of optimal sparse topology}

Let $d$ be a white stochastic disturbance with zero-mean and unit variance,
	\[
	\bE
	\left(
	d (t)
	\right)
	\, = \;
	0,
	~~
	\bE
	\left(
	d (t_1) \, d^T (t_2)
	\right)
	\, = \;
	I \, \delta (t_1 \, - \, t_2)
	\]
where $\bE$ is the expectation operator. The square of the ${\cal H}_2$ norm of the transfer function from $d$ to $\zeta$,
	 \be
	 \ba{rcl}
   	 \| H \|_2^2
	 & \!\! = \!\! &
	 \ds{\lim_{t \, \to \, \infty}}
	 \,
	 \bE
	 \left(
	 \psi^T (t) \, (Q \, + \, L_x \, R \, L_x) \, \psi (t)
	 \right)
	 \ea
	 \non
	 \ee
quantifies the steady-state variance amplification of closed-loop system~\eqref{eq.ss}. {As noted earlier,} the network average
	$
	\bar{\psi} (t)
	$
corresponds to the zero eigenvalue of the graph Laplacian and it is not observable from the performance output $\zeta$. Thus, the $\htwo$ norm is equivalently given by
    \be
	 \ba{rcl}
   	 \| H \|_2^2
   	 & \!\! = \!\! &
         \ds{\lim_{t \, \to \, \infty}}
         \,
	 \bE
	 \left(
	 \tpsi^T (t) \, (Q \, + \, L_x \, R \, L_x) \, \tpsi (t)
	 \right)
	  \\[0.25cm]
	 & \!\! = \!\! &
	 \trace
	 \left(
	 P
	 \,
	 (Q \, + \, L_x \, R \, L_x)
	 \right)
	 \, = \,
	 \inner{P}
	 {Q \, + \, L_x \, R \, L_x}
	 \ea
	 \non
	 \ee
where $\tpsi (t)$ is the vector of deviations of the states of individual nodes from $\bar{\psi} (t)$, 	
	 \be
	 \tpsi (t)
	 \; \DefinedAs \;
	 \psi (t)
	 \; - \;
	 \one \, \bar{\psi} (t)
	 \; = \;
	 \left(
	 I
	 \, - \,
	 (1/n)
	 \,
	 \one
	 \one^T
	 \right)
	 \psi (t)
	 \non
	 \ee
and $P$ is the steady-state covariance matrix of $\tpsi$,
	\beq
	P
	\; \DefinedAs \;
	 \ds{\lim_{t \, \to \, \infty}}
	 \,
	 \bE
	 \left(
	 \tpsi (t) \,
	 \tpsi^T (t)
	 \right).
	\non
	\eeq

The above measure of the amplification of stochastic disturbances is determined by {$\| H \|_2^2 = (1/2) J(x)$, where}
	{
	\be
    J (x)
    \, \DefinedAs \,
    \inner{
    \left(G_p
	\; + \;
	E \, \diag \left( x \right) E^T\right)^{-1}
    }
    {Q + L_x \, R \, L_x}.
    \label{eq.H2}
    \ee
    }
It can be shown that $J$ can be expressed as
	\beq
    \ba{rcl}
	J ( x )
	&\!\! {=} \!\!&
   	{\inner{\left(G_p
	\; + \;
	E \, \diag \left( x \right) E^T\right)^{-1}}{Q_p}
	\, + \,}
    \\[.15cm]
	&&
    \diag
	\left(
	E^{T} R \, E
	\right)^T
	\!
	x
    \,-\,
    \inner{R}{L_p}
    \, - \,
    1
    \ea
	\label{eq.J}
	\eeq
with
	\beq
	\ba{rcl}
	Q_p
	&
	\!\!
	 \DefinedAs
	 \!\!
	 &
	Q
    	\; + \;
    	(1/n)
    	\,
    	\one \one^{T}
	\; + \;
	L_p \, R \, L_p.
	\ea
	\non
	\eeq
{Note that the last two terms in~\eqref{eq.J} do not depend on the optimization variable $x$ and that the term $L_p\, R \,L_p$ in $Q_p$ has an interesting interpretation: it determines a state-weight that guarantees inverse optimality (in LQR sense) of $u = -L_p \psi$ for a system with no coupling between the nodes, $\dot{\psi} = u + d$.}

We formulate the design of a controller graph that provides an optimal tradeoff between the ${\cal H}_2$ performance of the closed-loop network and the controller sparsity as
    \beq
	\ba{rl}{
	\minimize\limits_{x}}
	&
	{J (x)
	\; +  \;
	\gamma \, \| x \|_1}
	\\[0.25cm]
{
	\subject}
    &
    {
    G_p
    \,  + \,
    E \, \diag \, ( x ) \, E^{T} \, \succ \,0}
	\ea
    \tag{SP}
	\label{eq.SP}
	\eeq
{where $J (x)$ and $G_p$ are given by \eqref{eq.J} and~\eqref{eq.Gb}, respectively.}	
The $\ell_1$ norm of $x$,
	$
	\| x \|_1
	\DefinedAs
	\sum_{l \, = \, 1}^{m} | x_l |,
	$
is introduced as a {convex} proxy for promoting sparsity. In~\eqref{eq.SP}, the vector of the edge weights $x \in \bbR^m$ is {the} optimization variable; the problem data are the positive regularization parameter $\gamma$, the state and control weights $Q$ and $R$, the plant graph Laplacian $L_p$, and the incidence matrix of the controller graph $E$.
	
The sparsity-promoting optimal control problem~\eqref{eq.SP} is a constrained optimization problem with a convex non-differentiable objective function~\cite{linfarjovALL12} {and a positive definite inequality constraint}. This implies convexity of~\eqref{eq.SP}. Positive definiteness of the strengthened graph Laplacian $G$ guarantees stability of the closed-loop network~\eqref{eq.ss1} on the subspace  $\one^{\bot}$, and thereby consensus in the absence of disturbances~\cite{mesege10}.


The consensus can be achieved even if some edge weights are negative~\cite{xiaboy04,xiaboykim07}. By expressing $x$ as a difference between two non-negative vectors $x_+$ and $x_{-}$,~\eqref{eq.SP} can be written as	
    \beq
	\!\!
	\ba{rl}
{
	\minimize\limits_{x_{+}, \, x_{-}}}
	&{
	\!\!
	\inner{\left(G_p
    \,  + \,
    E \, \diag \, ( x_+ \, - \, x_- ) \, E^{T}\right)^{-1}}{Q_p}
	\; +}
    \\[.15cm]
    &
    {
	(\gamma \, \one \, + \, c)^T x_{+}
	\, + \,
	(\gamma \, \one \, - \, c)^T x_{-}	}
	\\[0.25cm]
{
	\subject}
	&
{
	\!\!
    G_p
    \,  + \,
    E \, \diag \, ( x_+ \, - \, x_- ) \, E^{T}
    \, \succ \, 0}
    \\[.15cm]
    &
    \!\!
    {
    x_{+} \, \geq \,0,}
    ~~
    {
    x_{-} \, \geq \,0}
	\ea
	\label{eq.SP1}
	\eeq
where
	$
	c
	\DefinedAs
	\diag
	\left(
	E^{T} R \, E
	\right).
	$	
By utilizing the Schur complement,~\eqref{eq.SP1} can be cast to an SDP, and solved via standard IP method algorithms for small size networks.

\subsubsection*{Reweighted $\ell_1$ norm}

An alternative proxy for promoting sparsity is given by the weighted $\ell_1$ norm~\cite{canwakboy08},
	$
	\| w \circ  x \|_1
	\DefinedAs
	 \sum_{l \, = \, 1}^{m} w_l \, |x_l|
	$
where $\circ$ denotes elementwise product. The vector of non-negative weights $w \in \bbR^{m}$ can be selected to provide better approximation of non-convex cardinality function than the $\ell_1$ norm. An effective heuristic for weight selection is given by the iterative reweighted algorithm~\cite{canwakboy08}, with $w_l$ inversely proportional to the magnitude of $x_l$ in the previous iteration,
	\beq
	 w_l^+
	\, = \,
        1/( | x_l | \, + \, \varepsilon ).
	\label{eq.weights}
	\eeq
This puts larger emphasis on smaller optimization variables,  where a small positive parameter $\varepsilon$ ensures that $w_l^+$ is well-defined. If the weighted $\ell_1$ norm is used in~\eqref{eq.SP}, the vector of all ones $\one$ should be \mbox{replaced by the vector $w$ in~\eqref{eq.SP1}.}




	\vspace*{-3ex}
\subsection{Structured optimal control problem: debiasing step}
	\label{sec.polish}


After the structure of the controller graph Laplacian $L_x$ has been designed, {we fix the structure of $L_x$ and optimize the corresponding edge weights. This ``polishing'' or ``debiasing'' step is used to improve the performance relative to the solution of the regularized optimal control problem~\eqref{eq.SP}; see~\cite[Section 6.3.2]{boyvan04} for additional information. The structured optimal control problem is obtained by eliminating} the columns from the incidence matrix $E$ that correspond to zero elements in the vector of the optimal edge weights $x^\star$ {resulting from~\eqref{eq.SP}}. This yields a new incidence matrix $\hat{E}$ and leads to
	\beq
	\ba{rl}{
	\minimize\limits_{x}}
	&
{
	\inner{\left(G_p
    \,  + \,
    \hat{E} \, \diag \, ( x ) \, \hat{E}^{T}\right)^{-1}}{Q_p}
	\; + }
    \\[.15cm]
    &
    {
    \diag
	\left(
	\hat{E}^{T} R \, \hat{E}
	\right)^T
	\!
	x}
	\\[0.25cm]
{
	\subject}
    &
    {
    G_p
    \,  + \,
    \hat{E} \, \diag \, ( x ) \, \hat{E}^{T} \, \succ \,0}.
	\ea
    \non
	\eeq
{Alternatively,} this optimization problem is obtained by setting $\gamma = 0$ in~\eqref{eq.SP} and by replacing the incidence matrix $E$ with $\hat{E}$. {The solution} provides the optimal {vector of the edge weights $x$ for the} controller graph Laplacian with the desired structure.

	\vspace*{-2ex}
\subsection{Gradient and Hessian of $J (x)$}

We next summarize the first- and second-order derivatives of the objective function $J$, {given by~\eqref{eq.J}}, with respect to the vector of the edge weights $x$. The second-order Taylor series approximation of $J (x)$ around $\bar{x} \in \bbR^m$ is given by
	\[
	J ( \bar{x} + \tilde{x})
	\; \approx \;
	J ( \bar{x} )
	\; + \;
	\nabla J (\bar{x})^T \tilde{x}
	\; + \;
	\dfrac{1}{2}
	\,
	\tilde{x}^T
	\,
	\nabla^2 J (\bar{x})
	\,
	\tilde{x}.
	\]
For related developments we refer the reader to~\cite{ghoboysab08}.


    \begin{proposition}
The gradient and the Hessian of $J$ at $\bar{x}  \in \bbR^m$ are determined by
    \be
    \ba{rcl}
    \nabla J (\bar{x})
    & \!\! = \!\! &
    \; - \;
    \diag
    \left(
     E^T ( Y( \bar{x} )  \, - \, R) \, E
    \right)
    \\[0.15cm]
    \nabla^2 J (\bar{x})
    & \!\! = \!\! &
    H_1 ( \bar{x} ) \, \circ \, H_2 ( \bar{x} )
    \ea
    \non
    \ee
where
	\be
	\ba{rcl}
	Y (\bar {x})
	& \!\! \DefinedAs \!\! &
	\left(
	G_p
    	+
    	E \, D_{\bar{x}} \, E^{T}
	\right)^{-1}
	Q_p
	\left(
	G_p
    	+
    	E \, D_{\bar{x}} \, E^{T}
	\right)^{-1}
	\\[.15cm]
	H_1 ( \bar{x} )
	& \!\! \DefinedAs \!\! &
	E^{T} \, Y(\bar{x}) \, E
	\\[0.1cm]
	H_2 ( \bar{x} )
	& \!\! \DefinedAs \!\! &	
	E^{T}
	\left(
	G_p
     	\, + \,
     	E \, D_{\bar{x}} \, E^{T}
	\right)^{-1}
	E
	\\[0.15cm]
	D_{\bar{x}}
	& \!\! \DefinedAs \!\! &
	\diag \left( \bar{x} \right).
	\ea
	\non
	\ee	
	\label{prop.gradJ}
\end{proposition}

	\vspace*{-3ex}
\section{Dual problem}
	\label{sec.dual}
	
Herein, we study the Lagrange dual of the sparsity-promoting optimal control problem~\eqref{eq.SP1}, provide interpretation of dual variables, and construct dual feasible variables from primal feasible variables. Since minimization of the Lagrangian associated with~\eqref{eq.SP1} does not lead to an explicit expression for the dual function, we introduce an auxiliary variable $G$ and find the dual of
    \beq
	\!\!
	\ba{rl}
	\minimize\limits_{G, \; x_{\pm}}
	&
	\!\!
	\inner{G^{-1}}{Q_p}
	\, + \,
	(\gamma \, \one \, + \, c)^T x_{+}
	\, + \,
	(\gamma \, \one \, - \, c)^T x_{-}	
	\\[0.25cm]
	\subject
	&
	\!\!
    G
	\, - \,
    G_p
    \,  - \,
    E \, \diag \, ( x_+ \, - \, x_- ) \, E^{T}
    \, = \, 0
    \\[.15cm]
    &
    \!\!
    G \, \succ \, 0,
    ~~
    x_{+} \, \geq \,0,
    ~~
    x_{-} \, \geq \,0.
	\ea
    \tag{P}
	\label{eq.P}
	\eeq
{In~\eqref{eq.P}, $G$ represents the ``strengthened'' graph Laplacian of the closed-loop network and the equality constraint comes from~\eqref{eq.Ga}. As we show next, the Lagrange dual of the primal optimization problem~\eqref{eq.P} admits an explicit characterization.}


	\begin{proposition}
The Lagrange dual of the {primal optimization} problem~\eqref{eq.P} is given by
    \beq
	\ba{rl}
	\maximize\limits_{Y}
	&
	2 \, \trace
	\left(
	( Q_p^{1/2} \, Y \, Q_p^{1/2})^{1/2}
	\right)
	\, - \;
	\inner{Y}{G_p}
    \\[0.25cm]
    \subject
	&
	\|\, \diag \left( E^{T} (Y \,-\, R) \, E \right) \|_{\infty}
	\, \leq \, \gamma
	\\[0.15cm]
	&
	Y \; \succ \; 0,
    ~~   	
	Y \, \one \;=\; \one
	\ea
       \tag{D}
	\label{eq.D}
	\eeq
where $Y = Y^T \in \bbR^{n \times n}$ is the dual variable associated with the equality constraint in~\eqref{eq.P}. The duality gap is	
	\be
    \eta
    \, = \,
    y_+^T \, x_+
    \; + \;
    y_-^T \, x_-
	\, = \,
    \one^T
    \left(
    y_+ \circ \, x_+
    \; + \;
    y_- \circ \, x_-
    \right)
    \label{eq.etaG}
    \ee
where
	 \begin{subequations}
	\begin{align}
        	y_+
	& ~ = ~
	\gamma \, \one
    \; - \;
     \diag
	\left(
	E^T ( Y - R) \, E
	\right)
    \;  \geq \; 0
	\label{eq.min1}
    \\[0.1cm]
    y_-
	& ~ = ~
	\gamma \, \one
   	\; + \;
     \diag
	\left(
	E^T ( Y - R) \, E
	\right)
    \;  \geq \; 0.
	\label{eq.min2}
	\end{align}
	\label{eq.min}
	\end{subequations}
\!\!\!\!are the Lagrange multipliers associated with elementwise inequality constraints in~\eqref{eq.P}.		
	\end{proposition}
	
	
	\begin{proof}
 The Lagrangian of~\eqref{eq.P} is given by
 	\be
	\!\!
    \ba{rcl}
    \mathcal{L}
    & \!\! = \!\! &
    \inner{G^{-1}}{Q_p}
	\, + \,
	 \inner{Y}{G}
	 \,-\,
	 \inner{Y}{G_p} ~+
	\\[0.15cm]
        & \!\! \!\! &
	\left(
	\gamma \, \one
    	\, - \,
    	\diag
	\left(
	E^T ( Y - R) \, E
	\right)
	\, - \,
	y_{+}
	\right)^T		
	 x_{+}
	 ~ +
	 \\[0.15cm]
	 & \!\! \!\! &
	 \left(
	\gamma \, \one
    	\, + \,
    	 \diag
	\left(
	E^T ( Y - R) \, E
	\right)
	\, - \,
	y_{-}
	\right)^T		
	 x_{-}.
	    \ea
    \label{eq.L}
    \ee
Note that no Lagrange multiplier is assigned to the positive definite constraint on $G$ in $\mathcal{L}$. Instead, we determine conditions on $Y$ and $y_{\pm}$ that guarantee $G \succ 0$.

Minimizing $\mathcal{L}$ with respect to $G$ yields
 	\begin{subequations}
	\beq
	G^{-1}
	\,
	Q_p
	\,
	G^{-1}
	\, = \;
	Y
	\label{eq.Ya1}
	\eeq
or, equivalently,
	\beq
	G
	\; = \;
	Q_p^{1/2}
	\left( Q_p^{1/2} \, Y \, Q_p^{1/2} \right)^{-1/2}
	Q_p^{1/2}.
	\label{eq.Yb1}
	\eeq	
	\label{eq.Y1}
	\end{subequations}	
\!\!Positive definiteness of $G$ and $Q_p$ implies $Y \succ 0$. Furthermore, since $Q_p \one = \one$, from~\eqref{eq.Gc} and~\eqref{eq.Ya1} we have
	\[
	Y \, \one \; = \; \one.
	\]
Similarly, minimization with respect to $x_+$ and $x_-$ leads to~{\eqref{eq.min1} and \eqref{eq.min1}}. Thus, non-negativity of $y_+$ and $y_-$ amounts to
       \[
    -\gamma \, \one
    \, \leq \,
	\diag
	\left(
	E^T ( Y - R) \, E
	\right)    \, \leq \,
    \gamma \, \one
    \]
or, equivalently,
    \[
    \|
    \,
    \diag
    \left(
	E^T ( Y - R) \, E
	\right)
       \|_\infty
    \, \leq \,
    \gamma.
    \]
Substitution of~\eqref{eq.Y1} and~\eqref{eq.min} into~\eqref{eq.L} eliminates $y_{+}$ and $y_{-}$ from the dual problem. We can thus represent the dual function,
	$
	\inf_{G, \, x_{\pm}}
	\mathcal{L}(G, x_{\pm}; Y, y_{\pm}),
	$
as
    \[
    2 \, \trace
	\left(
	( Q_p^{1/2} \, Y \, Q_p^{1/2})^{1/2}
	\right)
	\, - \;
	\inner{Y}{G_p}
    \]
which allows us to bring the dual of~\eqref{eq.P} to~\eqref{eq.D}.
	\end{proof}
	
Any dual feasible $Y$ can be used to obtain a lower bound on the optimal value of the primal problem~\eqref{eq.P}. Furthermore, the difference between the objective functions of the primal (evaluated at the primal feasible $(G,x_{\pm})$) and dual (evaluated at the dual feasible $Y$) problems yields {expression~\eqref{eq.etaG} for the duality gap $\eta$}, where $y_+$ and $y_-$ are given by~\eqref{eq.min1} and~\eqref{eq.min2}. {The duality gap} can be used to estimate distance to optimality.

Strong duality follows from {Slater's theorem~\cite{boyvan04}}, i.e., convexity of the primal problem~\eqref{eq.P} and strict feasibility of the constraints in~\eqref{eq.P}. This implies that at optimality, the duality gap $\eta$ for the primal problem~\eqref{eq.P} and the dual problem~\eqref{eq.D} is zero. Furthermore, if ($G^\star, x^\star_{\pm}$) are optimal points of~\eqref{eq.P}, then
	$
	Y^\star
	=
	(G^\star)^{-1}
	Q_p
	\,
	(G^\star)^{-1}
	$
is the optimal point of~\eqref{eq.D}. Similarly, if $Y^\star$ is the optimal point of~\eqref{eq.D},
	\[
	G^\star
	\; = \;
	Q_p^{1/2}
	\left( Q_p^{1/2} Y^\star Q_p^{1/2} \right)^{-1/2}
	Q_p^{1/2}
	\]
is the optimal point of~\eqref{eq.P}. The optimal vector of the edge weights $x^\star$ is determined by the non-zero off-diagonal elements of the controller graph Laplacian, $L_x^\star = G^\star - G_p$.

	\subsubsection*{Interpretation of dual variables}
	For electrical networks, the dual variables have appealing interpretations. Let $\iota \in \mathds{R}^n$ be a random current injected into the resistor network satisfying
     \[
     \one^T \iota
     \; = \; 0,
     ~~~
     \bE
     \left( \iota \right)
     \, = \; 0,
     ~~~
     \bE
     \left( \iota \iota^T \right)
     \, = \;
     Q
	\; + \;
	L_p \, R \, L_p.
     \]
     The vector of voltages $\vartheta \in \mathds{R}^m$ across the edges of the network is then given by $\vartheta = E^T  G^{-1} \iota $. Furthermore, since
     \be
     \ba{rcl}
     \bE
     \left( \vartheta\,\vartheta^T \right)
     &\!\!=\!\!&
     E^T \, G^{-1}\,\bE \left(  \iota \iota^T \right) G^{-1}\,E	
     \; = \;
     E^T\, Y\,E,
     \ea
     \non
     \ee
the dual variable $Y$ is related to the covariance matrix of voltages across the edges. Moreover,~\eqref{eq.min} implies that $y_+$ and $y_-$ quantify the deviations between variances of edge voltages from their respective upper and lower bounds.

	\begin{remark}
\label{stopp}
For a primal feasible $x$, $Y$ resulting from~\eqref{eq.Ya1} with $G$ given by~\eqref{eq.Ga} may not be dual feasible. Let
	\begin{subequations}
	\label{eq.Yhat}
	\be
    \hat{Y}
    \; \DefinedAs \;
    \beta \,Y
    \; + \;
    \dfrac{1\,-\,\beta}{n} \, \one \one^T
    \label{eq.Yhat1}
    \ee
and let the control weight be $R = r\, I$ with $r > 0$. If
    \be
    \beta
    \;\leq\;
    \dfrac{\gamma \,+\, 2\,r}{ \|\,\diag \left( E^{T} (Y \,-\, R) \, E\right)\|_\infty \,+\, 2\,r}
    \label{eq.Yhat2}
    \ee
    \end{subequations}
then $\hat{Y}$ satisfies the inequality constraint in~\eqref{eq.D} and it is thus dual feasible.
		\end{remark}

	\vspace*{-2ex}
\section{Customized algorithms}
	\label{sec.IPgen}

We next exploit the structure of the sparsity-promoting optimal control problem~\eqref{eq.SP} and develop customized algorithms based on the proximal gradient and Newton methods. The proximal gradient algorithm is a first-order method that uses a simple quadratic approximation of $J$ in~\eqref{eq.SP}. This yields an explicit update of the vector of the edge weights via application of the soft-thresholding operator. In the proximal Newton method a sequential quadratic approximation of the smooth part of the objective function in~\eqref{eq.SP} is used and the search direction is efficiently computed via cyclic coordinate descent over the set of active variables.	
	
	\vspace*{-2ex}
\subsection{Proximal gradient method}
	\label{sec.prox1}

We next use the proximal gradient method to solve~\eqref{eq.SP}. A simple quadratic approximation of $J (x)$ around the current iterate $x^k$,	
	\be
	J ( x )
	\; \approx \;
	J ( x^k )
	\; + \;
	\nabla J ( x^k )^T ( x \, - \, x^k )
	\; + \;
	\dfrac{1}{2 \alpha_k}
	\,
	\norm{ x \, - \, x^k }_2^2
	\non
    	\ee
is substituted to~\eqref{eq.SP} to obtain
	\beq
	x^{k+1}
	\; = \;
	\argmin\limits_{x}
	\;
	g (x)	
	\, + \,
	\dfrac{1}{2 \alpha_k}
	\,
	\norm{ x \, - \, ( x^k \, - \, \alpha_k \nabla J ( x^k ) )}_2^2.
    \non
	\eeq
Here, $\alpha_k$ is the step-size and the update is determined by the proximal operator of the function $\alpha_k \, g$,
	\beq
	x^{k+1}
	\: = \;
	\prox_{\alpha_k g}
	\!
	\left(
	x^k \, - \, \alpha_k \nabla J ( x^k )
	\right).
	\non
	\eeq 	
In particular, for $g (x) = \gamma \, \|x\|_1$, we have
	\beq
	x^{k+1}
	\: = \;
    	{\cal S}_{\gamma \alpha_{k}}
	\!
	\left(
	x^k \, - \, \alpha_{k} \nabla J ( x^k )
	\right)
	\non
	\eeq
where ${\cal S}_{\kappa}(y) = \sign \, (y) \max \, (|y| - \kappa,0)$ is the soft-thresholding function.


The proximal gradient algorithm converges with rate $O(1/k)$ if $\alpha_k < 1/L$, where $L$ is the Lipschitz constant of $\nabla J$~\cite{becteb09,parboy13}. It can be shown that $\nabla J$ is Lipschitz continuous but, since it is challenging to explicitly determine $L$, we adjust $\alpha_k$ via backtracking. To provide a better estimate of $L$, we initialize $\alpha_k$ using the Barzilai-Borwein (BB) method which provides an effective heuristic for approximating the Hessian of the function $J$ via the scaled version of the identity~\cite{barbor88}, $(1/\alpha_k) I$. At the $k$th iteration, the initial BB step-size $\alpha_{k,0}$,
	\beq
	\alpha_{k,0}
	\; \DefinedAs \;
	\frac{\norm{x^{k} \, - \, x^{k-1}}_2^2}
	{(x^{k-1} \, - \, x^{k})^T \, ( \nabla J(x^{k-1}) \, - \,  \nabla J(x^{k}) )}
	\label{eq.BBstep}
	\eeq
is adjusted via backtracking until the inequality constraint in~\eqref{eq.SP} is satisfied and
	\be
	J ( x^{k+1} )
	\leq
	J ( x^k )
	+
	\nabla J ( x^k )^T ( x^{k+1} - x^k )
	+
	\dfrac{1}{2 \alpha_k}
	\norm{ x^{k+1} - x^k }_2^2.
	\non
    \ee
Since $J$ is continuously differentiable with Lipschitz continuous gradient, this inequality holds for any $\alpha_k< 1/L$ and the algorithm converges sub-linearly~\cite{becteb09}. This condition guarantees that objective function decreases at every iteration. Our numerical experiments in Section~\ref{sec.examples} suggest that BB step-size initialization significantly enhances the rate of convergence.
	
\begin{remark}
	The biggest computational challenge comes from evaluation of the objective function and its gradient. Since the inverse of the strengthened graph Laplacian $G$ has to be computed, with direct computations these evaluations take $O(n^3)$ and $O(n m^2)$ flops, respectively. However, by exploiting the problem structure, $\nabla J$ can be computed more efficiently. The main cost arises in the computation of $ \diag\,(E^T Y E)$. We instead compute it using $\text{sum}\, (E^T\circ(Y E))$ which takes $O(n^2 m)$ operations. Here, $\text{sum}\,(A)$ is a vector which contains summation of each row of the matrix $A$ in its entries. For networks with $m\gg n$ this leads to significant speed up. Moreover, in contrast to direct computation, we do not need to store the $m \times m$ matrix $E^T Y E$. Only formation of the columns is required which offers memory saving.
\end{remark}

	\vspace*{-2ex}
\subsection{Proximal Newton method}
	\label{sec.cd1}


In contrast to the proximal gradient algorithm, the proximal Newton method benefits from second-order Taylor series expansion of the smooth part of the objective function in~\eqref{eq.SP}. Herein, we employ cyclic coordinate descent over the set of active variables to efficiently compute the Newton direction.


By approximating the smooth part of the objective function $J$ in~\eqref{eq.SP} with the second-order Taylor series expansion around the current iterate $\bar{x}$,
	\be
	J ( \bar{x} + \tilde{x})
	\; \approx \;
	J ( \bar{x} )
	\; + \;
	\nabla J (\bar{x})^T
	\,
	\tilde{x}
	\; + \;
	\dfrac{1}{2}
	\,
	\tilde{x}^T
	\,
	\nabla^2 J (\bar{x})
	\,
	\tilde{x}
    \non
    \ee
the problem~\eqref{eq.SP} becomes
    \be
    \ba{cl}
    \minimize\limits_{\tilde{x}}
    &
   \nabla J (\bar{x})^T
   \,
   \tilde{x}
	\; + \;
	\dfrac{1}{2}
	\,
	\tilde{x}^T
	\,
	\nabla^2 J (\bar{x})
	\,
	\tilde{x}
    \,+\,
    \gamma
    \,
    \|
    \bar{x}
    \, + \,
    \tilde{x}\|_1
	\\[0.25cm]
    \subject
    &
    G_p
    \,+\,
    E \,\diag\,(\bar{x}
    \, + \,
    \tilde{x}) E^T
    \; \succ\;
    0.
    \ea
	\label{eq.SP3}
    \ee
Let $\tilde{x}$ denote the current iterate approximating the Newton direction. By perturbing $\tilde{x}$ in the direction of the $i$th standard basis vector $\mre_i$ in $\bbR^m$,
the objective function in~\eqref{eq.SP3} becomes
	\be
        \ba{rcl}
	\nabla J (\bar{x})^T
	\!
   	\left(
   	\tilde{x}
	\, + \,
	\delta_i \, \mre_i
	\right)
	& \!\!\! + \!\!\! &
	\dfrac{1}{2}
	\left(
   	\tilde{x}
	\, + \,
	\delta_i \, \mre_i
	\right)^T
	\nabla^2 J (\bar{x})
	\left(
   	\tilde{x}
	\, + \,
	\delta_i \, \mre_i
	\right)
    \\[.15cm]
    & \!\!\! + \!\!\! &
    \gamma
    \,
    |
    \bar{x}_i
    \, + \,
    \tilde{x}_i\,+\,
    \delta_i|.
    \ea
	\non
	\eeq
Elimination of constant terms allows us to bring~\eqref{eq.SP3} into
	\be
    \ba{cl}
    \minimize\limits_{\delta_i}
    &
    \dfrac{1}{2}
    \,
    a_i
    \,
    \delta_i^2
    \; + \;
    b_i
    \,
    \delta_i
    \;+\;
    \gamma
    \,
    |
    c_i\,+\,
    \delta_i|
    \ea
    \label{eq.NewtonDirection2}
    \ee
where the optimization variable is the scalar $\delta_i$ and ($a_i$, $b_i$, $c_i$, $\bar{x}_i$, $\tilde{x}_i$) are the problem data with
	\beq
	\ba{rrl}
	a_i
	& \!\! \DefinedAs \!\! &
	\mre_i^T\, \nabla^2 J(\bar{x})\,\mre_i
	\\[0.15cm]
	b_i
	& \!\! \DefinedAs \!\! &
    \left(
    \nabla^2 J(\bar{x})
	\,
	\mre_i
	\right)^T
	\!
	\tilde{x}
    \; + \;
	\mre_i^T
	\,
    \nabla J (\bar{x})
    \\[0.15cm]
	c_i
	& \!\! \DefinedAs \!\! &
    \bar{x}_i\,+\,\tilde{x}_i.
	\ea
	\non
	\eeq
The explicit solution to~\eqref{eq.NewtonDirection2} is given by
    \beq
	\delta_i
	\; = \;
	-\,c_i
    \;+\;
    {\cal{S}}_{\gamma/a_i}
    \!
    \left(
    c_i \,-\,b_i/a_i
    \right).
	\non
	\eeq


After the Newton direction $\tilde{x}$ has been computed, we determine the step-size $\alpha$ via backtracking. This guarantees positive definiteness of the strengthened graph Laplacian and sufficient decrease of the objective function. We use generalization of Armijo rule~\cite{tseyun09} to find an appropriate step-size $\alpha$ such that $G_p + E ~\! \diag(\bar{x} + \alpha \tilde{x}) E^T$ is positive definite matrix and
    \be
    \ba{lcl}
	J ( \bar{x} + \alpha \tilde{x} )
    \,+\,
    \gamma\,\| \bar{x} + \alpha \tilde{x} \|_1
	\, \leq \,
	J ( \bar{x} )
    \,+\,
    \gamma\,\| \bar{x} \|_1
	\; + \;&&
    \\[.15cm]
    \alpha\,\sigma
    \left(
	\nabla J ( \bar{x} )^T \tilde{x}
    \,+\,
    \gamma\,\| \bar{x}+ \tilde{x} \|_1
    \,-\,
    \gamma\,\| \bar{x} \|_1
    \right).
    &&
    \ea
	\non
    \ee



		
\begin{remark}
The parameter $a_i$ in~\eqref{eq.NewtonDirection2} is determined by the $i$th diagonal element of the Hessian $\nabla^2 J(\bar{x})$. On the other hand, the $i$th column of $\nabla^2 J(\bar{x})$ and the $i$th element of the gradient vector $\nabla J (\bar{x})$ enter into the expression for $b_i$. All of these can be obtained directly from $\nabla^2 J(\bar{x})$ and $\nabla J (\bar{x})$ and forming them does not require any multiplication. Computation of a single vector inner product between the $i$th column of the Hessian and $\tilde{x}$ is required in $b_i$, which typically takes $O(m)$ operations. To avoid direct multiplication, in each iteration after finding $\delta_i$, we update the vector $\nabla^2 J(\bar{x})^T \tilde{x}$ using the correction term $\delta_i (E^T Y E_i)\circ((G^{-1} E_i)^T E)^T$ and take its $i$th element to form $b_i$. Here, $E_i$ is the $i$th column of the incidence matrix of the controller graph. This also avoids the need to store the Hessian of $J$, which is an $m \times m$ matrix, thereby leading to a significant memory saving.
\end{remark}


	\begin{remark}
Active set strategy is an effective means for determining the directions that do not need to be updated in the coordinate descent algorithm. At each outer iteration, we classify the variable as either active or inactive based on the values of $\bar{x}_i$ and the $i$th component of the gradient vector $\nabla J (\bar{x})$. For $g (x) = \gamma \, \|x\|_1$, the $i$th \mbox{search direction is inactive if}
	\[
	\bar{x}_i
	\; = \;
	0
	~~
	\mathrm{and}
	~~
	|\,\mre_i^T
	\,
	\nabla J ( \bar{x} )\,|
	\, < \,
	\gamma\,-\,\epsilon
	\]
and it is active otherwise. Here, $\epsilon > 0$ is a small number (e.g., $\epsilon = 0.0001 \gamma$). The Newton direction is then obtained by solving the optimization problem over the set of active variables. This significantly improves algorithmic efficiency for large values of the regularization parameter $\gamma$.
	\end{remark}

\subsubsection*{Convergence analysis}

In~\eqref{eq.SP}, $J(x)$ is smooth for $G_p + E \,\diag(x)\, E^T \succ 0$ and the non-smooth part is given by the $\ell_1$ norm of $x$. The objective function of the form $J(x) + g(x)$ was studied in~\cite{hsisusdhirav14}, where $J$ is smooth over the positive definite cone and $g$ is a separable non-differentiable function. Theorem 16 from~\cite{hsisusdhirav14} thus implies super-linear (i.e., quadratic) convergence rate of the quadratic approximation method for~\eqref{eq.SP}.

    \subsubsection*{Stopping criteria}

     The norms of the primal and dual residuals $r_p$ and $r_d^\pm$ as well as the duality gap $\eta$ are used as stopping criteria. In contrast to the stopping criteria available in the literature, this choice enables fair comparison of the algorithms. We use~\eqref{eq.Yhat} to construct a dual feasible $\hat{Y}$ and obtain $y_\pm$ from~\eqref{eq.min}, {\eqref{eq.etaG} to compute the duality gap $\eta$, and}
    {
    \be
    \!
    \ba{rcl}
    r_p (x, \, x_{\pm})
    &
    \!\! \DefinedAs \!\!
    &
    x
    \,-\,
    x_+
    \,+\,
    x_-
    \\[.1cm]
    r_d^+ (x, \, y_+)
    &
    \!\! \DefinedAs \!\!
    &
    \gamma
    \,
    \one
    \,-\,
    \diag
	\left(
	E^T (\hat{Y} - R) \, E
	\right)
    \,-\,
    y_+
    \\[.15cm]
    r_d^-(x, \, y_-)
    &
    \!\! \DefinedAs \!\!
    &
    \gamma
    \,
    \one
    \,+\,
    \diag
	\left(
	E^T ( \hat{Y} - R) \, E
	\right)
    \,-\,
    y_-.
    \ea
    \non
    \ee
to determine the primal and dual residuals.}
{
\subsubsection*{Comparison of algorithms}
Table~\ref{tab.comparison} compares and contrasts features of our customized proximal algorithms and the algorithm based on the primal-dual IP method developed in~\cite{mogjovACC15}.}
	
     \begin{table*}
    \centering
    \caption{{Comparison of our customized proximal algorithms with the primal dual IP method of~\cite{mogjovACC15}.}}
    \label{tab.comparison}
    \begin{tabular}{|c|c|c|c|c|c|}
      \hline
      Algorithm &  primal-dual IP method & proximal gradient & proximal Newton \\
      \hline\hline
      Order& $2$nd & $1$st & $2$nd\\
      \hline
      Search direction  & PCG & explicit update& coordinate descent \\
      \hline
      Speed-up strategy&  PCG with preconditioner & BB step-size initialization & active set strategy \\
      \hline
      Memory  &  no storage of $m \times m$ matrices &  no storage of $m \times m$ matrices &   no storage of $m \times m$ matrices \\
      \hline
      Most expensive part &  $O(m^3)$   & $O(n^2 m)$ &  $O(m^2)$  \\
      \hline
      Convergence rate  &  super-linear & linear & super-linear (quadratic) \\
      \hline
    \end{tabular}
    \end{table*}

	\vspace*{-2ex}
\section{Growing connected resistive networks}
	\label{sec.resistive}
	
The problem of optimal {topology design for} stochastically-forced networks has many interesting variations. An important class is given by resistive networks in which all edge weights are non-negative, $x \geq 0$. Here, we study the problem of growing connected resistive networks; {e.g., see~\cite{ghoboy06}}. In this, the plant graph is connected and there are {\em no joint edges\/} between the plant and the controller graphs. Our objective is to enhance the closed-loop performance by adding a small number of edges. As we show below, inequality constraints in this case amount to non-negativity of controller edge weights. This simplifies optimality conditions and enables further improvement of the computational efficiency of \mbox{our customized algorithms.}


	

The restriction on connected plant graphs implies positive definiteness of the strengthened graph Laplacian of the plant,
	$
	G_p
	=
	L_p  + (1/n) \, \one \one^T
	\succ
	0.
	$
Thus,
	$
	G_p
    +
    E \, \diag \, ( x ) \, E^{T}$ is always positive definite for connected resistive networks and~\eqref{eq.SP} simplifies to
	\beq
	\ba{rl}
	\minimize\limits_{x}
	&
	f (x)	
	\; + \;
	g (x)
	\ea
	\label{eq.SP2}
	\eeq
where
	\beq
	f (x)
	\; \DefinedAs \;
	J (x)
	\; + \;
	\gamma \, \one^T x
	\non
	\eeq
and $g (x)$ is the indicator function for the non-negative orthant,
	\[
	g(x)
	\; \DefinedAs \;
	I_+ (x)
	\; = \;
	\left\{
	\ba{rl}
	0,
	&
	x \, \geq \, 0
	\\[0.1cm]
	+\infty,
	&
	\mathrm{otherwise}.
	\ea
	\right.
	\]

{As in Section~\ref{sec.dual}, in order to determine the Lagrange dual of the optimization problem~\eqref{eq.SP2}, we introduce an additional optimization variable $G$ and rewrite~\eqref{eq.SP2} as}
	\beq
	\ba{rl}
    {
	\minimize\limits_{G, \,x}}
	&
	{\inner{G^{-1}}{Q_p}
	\, + \,
	(
	\gamma \, \one
    \,+\,
    \diag
	\left(
	E^{T} R \, E
	\right)
	)^T
    x}
	\\[0.25cm]
    {
	\subject}
	&
    {
    G
	\, - \,
    G_p
    \,  - \,
    E \, \diag \, (x) \, E^{T}
    \, = \, 0}
    \\[.15cm]
    &{
	x \, \geq \, 0.}
	\ea
    \tag{P1}
	\label{eq.P1}
	\eeq

\begin{proposition}
    The Lagrange dual of the {primal optimization} problem~\eqref{eq.P1} is given by
	\beq
	\ba{rl}
	\!\!\!\!
	\maximize\limits_{Y}
	&
	2 \, \trace
	\left(
	( Q_p^{1/2} \, Y \, Q_p^{1/2})^{1/2}
	\right)
	\, - \;
	\inner{Y}{G_p}
    \\[0.25cm]
	\!\!\!\!
    \subject
	&
	\diag \left( E^{T} (Y \,-\, R) \, E \right)
	\; \leq \;
	\gamma \, \one
	\\[0.2cm]
	\!\!\!\!
	&
	Y \; \succ \; 0,
    	~~   	
	Y \, \one \;=\; \one
	\ea
   	\tag{D1}
	\label{eq.D1}
	\eeq
where $Y$ is the dual variable associated with the equality constraint in~\eqref{eq.P1}. The duality gap is
	\be
	\eta
	\; = \;
	y^T x
	\; = \;
	\one^T
	(  y \, \circ \, x )
	\label{eq.eta}
	\ee
where
	\be
	y
	\; \DefinedAs \;
	\gamma \, \one
    	\; - \;
     	\diag
	\left(
	E^T ( Y - R) \, E
	\right)
    	\;  \geq \; 0
	\label{eq.y}
	\ee
represents the dual variable associated with the non-negativity constraint on the vector of the edge weights $x$.
    \end{proposition}

	

\begin{remark}
\label{stopp2}
For connected resistive networks with the control weight $R = r\, I$, $\hat{Y}$ given by~\eqref{eq.Yhat1} is dual feasible if
    \be
    \beta
    \;\leq\;
    \dfrac{\gamma \,+\, 2\,r}{\max\left( \diag \,( E^{T} (Y \,-\, R) \, E) \right) \,+\, 2\,r}.
    \label{eq.beta}
    \ee
\end{remark}

	\vspace*{-1ex}
\subsection{Proximal gradient method}
	\label{sec.prox}

Using a simple quadratic approximation of the smooth part of the objective function $f$ around the current iterate $x^k$
	\be
	f ( x )
	\; \approx \;
	f ( x^k )
	\; + \;
	\nabla f ( x^k )^T ( x \, - \, x^k )
	\; + \;
	\dfrac{1}{2 \alpha_k}
	\,
	\norm{ x \, - \, x^k }_2^2
	\non
    	\ee
the optimal solution of~\eqref{eq.SP2} is determined by the proximal operator of the function $g (x) = I_+ (x)$,
	\beq
	x^{k+1}
	\: = \;
	\left(
	x^k \, - \, \alpha_k \nabla f ( x^k )
	\right)_+
	\non
	\eeq
where $(\cdot)_+$ is the projection on the non-negative orthant. Thus, the action of the proximal operator is given by the projected gradient.

As in Section~\ref{sec.prox1}, we initialize $\alpha_k$ using the BB heuristics but we skip the backtracking step here and employ a non-monotone BB scheme~\cite{daifle05,wrinowfig09}. The effectiveness of this strategy has been established on quadratic problems~\cite{barbor88,daifle05}, but its convergence in general is hard to prove. In Section~\ref{sec.examples}, we demonstrate efficiency of this approach.


	\vspace*{-2ex}
\subsection{{Proximal Newton method}}
	\label{sec.cd}

We next adjust the customized algorithm based on proximal Newton method for growing connected resistive networks. We approximate the smooth part of the objective function $f$ in~\eqref{eq.SP2} using the second-order Taylor series expansion around the current iterate $\bar{x}$,
	\be
	f ( \bar{x} + \tilde{x})
	\; \approx \;
	f ( \bar{x} )
	\; + \;
	\nabla f (\bar{x})^T
	\,
	\tilde{x}
	\; + \;
	\dfrac{1}{2}
	\,
	\tilde{x}^T
	\,
	\nabla^2 f (\bar{x})
	\,
	\tilde{x}
    \non
    \ee
and rewrite~\eqref{eq.SP2} as
    \be
    \ba{cl}
    \minimize\limits_{\tilde{x}}
    &
   \nabla f (\bar{x})^T
   \,
   \tilde{x}
	\; + \;
	\dfrac{1}{2}
	\,
	\tilde{x}^T
	\,
	\nabla^2 f (\bar{x})
	\,
	\tilde{x}
	\\[0.25cm]
    \subject
    &
    \bar{x}
    \; + \;
    \tilde{x}
    \; \geq \;
    0.
    \ea
	\label{eq.SP4}
    \ee
By perturbing $\tilde{x}$ in the direction of the $i$th standard basis vector $\mre_i$ in $\bbR^m$,
	$
	\tilde{x}
	+
	\delta_i \, \mre_i,
	$
the objective function in~\eqref{eq.SP4} becomes
	\be
	\nabla f (\bar{x})^T
	\!
   	\left(
   	\tilde{x}
	\, + \,
	\delta_i \, \mre_i
	\right)
	\; + \;
	\dfrac{1}{2}
	\left(
   	\tilde{x}
	\, + \,
	\delta_i \, \mre_i
	\right)^T
	\nabla^2 f (\bar{x})
	\left(
   	\tilde{x}
	\, + \,
	\delta_i \, \mre_i
	\right).
	\non
	\eeq
Elimination of constant terms allows us to bring~\eqref{eq.SP4} into
	\be
    \ba{cl}
    \minimize\limits_{\delta_i}
    &
    \dfrac{1}{2}
    \,
    a_i
    \,
    \delta_i^2
    \; + \;
    b_i
    \,
    \delta_i
     \\[0.25cm]
    \subject
    &
    \bar{x}_i
    \; + \;
    \tilde{x}_i
    \; + \;
    \delta_i
    \; \geq \;
    0.
    \ea
    \label{eq.NewtonDirection3}
    \ee
The optimization variable is the scalar $\delta_i$ and $a_i$, $b_i$, $\bar{x}_i$, and $\tilde{x}_i$ are the problem data with
	\beq
	\ba{rrl}
	a_i
	& \!\! \DefinedAs \!\! &
	\mre_i^T\, \nabla^2 f(\bar{x})\,\mre_i
	\\[0.15cm]
	b_i
	& \!\! \DefinedAs \!\! &
    	\left(
    	\nabla^2 f(\bar{x})
	\,
	\mre_i
	\right)^T
	\!
	\tilde{x}
    	\; + \;
	\mre_i^T
	\,
    	\nabla f (\bar{x})
	\ea
	\non
	\eeq
	
The explicit solution to~\eqref{eq.NewtonDirection3} is given by
    \beq
	\delta_i
	\; = \;
	\left\{
	\ba{rl}
	-b_i/a_i,
	&
	\bar{x}_i
	\,+\,
	\tilde{x}_i
	\, - \,
	b_i/a_i
	\, \geq\,
	0
	\\[0.15cm]
	-\left( \bar{x}_i \,+\,\tilde{x}_i \right),
	&
	\otherwise.
	\ea
	\right.
	\non
	\eeq	

	
After the Newton direction $\tilde{x}$ has been computed, we determine the step-size $\alpha$ via backtracking. This guarantees positivity of the updated vector of the edge weights, $\bar{x}+ \alpha \tilde{x}$, and sufficient decrease of the objective function,
	$
	f( \bar{x} + \alpha \tilde{x} )
	\leq
	f ( \bar{x} )
        +
    	\alpha\,\sigma\,
	\nabla f ( \bar{x} )^T \tilde{x}.
    $	

	\begin{remark}
As in Section~\ref{sec.cd1}, we use an active set strategy to identify the directions that do not need to be updated in the coordinate descent algorithm. For $g (x) = I_+ (x)$, the $i$th search direction is inactive if
	$
	\{
	\bar{x}_i
	=
	0
	~
	\mathrm{and}
	~
	\mre_i^T
	\,
	\nabla f ( \bar{x} )
	\geq
	0
	\}
	$
and it is active otherwise.
	\end{remark}
	

\subsubsection*{Stopping criteria}

The norm of the dual residual, $r_d$, and the duality gap, $\eta$, are used as stopping criteria. The dual variable $y$ is obtained from~\eqref{eq.y} where $\hat{Y}$ is given by~\eqref{eq.Yhat1} and $\beta$ satisfies~\eqref{eq.beta}. {At each iteration, $\eta$ is evaluated using~\eqref{eq.eta} and the dual residual is determined by
	\be
	r_d (x, y)
	\; \DefinedAs \;
	\gamma
	\,
	\one
	\; - \;
	\diag \left( E^{T} (Y(x) \,-\, R) \, E \right)
	\;  - \;
	y.
	\non
	\ee}

	\vspace*{-5ex}
\section{Computational experiments}
	\label{sec.examples}

We next provide examples and evaluate performance of our customized algorithms. Algorithm proxBB represents proximal gradient method with BB step-size initialization and proxN identifies proximal Newton method in which the search direction is found via coordinate descent. {Performance is compared with the PCG-based primal-dual IP method of~\cite{mogjovACC15} and the greedy algorithm of~\cite{sumshalygdor15}.} We have implemented all algorithms in {\sc Matlab} and executed tests on a 3.4 GHz Core(TM) i7-3770 Intel(R) machine with 16GB RAM.


In all examples, we set $R = I$ and choose the state weight that penalizes the mean-square deviation from the network average,
	$
	Q
	=
	I
	-
	(1/n)
    \,
    \one \one^{T}.
	$
The absolute value of the dual residual, $r_d$, and the duality gap, $\eta$, are used as stopping criteria. We set the tolerances for $r_d$ and $\eta$ to $10^{-3}$ and $10^{-4}$, respectively. Finally, for connected plant networks
	\beq
    	\gamma_{\max}
	\; \DefinedAs \;
	\| \, \diag\,( E^{T} \, G_p^{-1} \,Q\, G_p^{-1}\, E )\,  \| _{\infty}
	\label{eq.gamma_max}
	\non
	\eeq	
identifies the value of the regularization parameter $\gamma$ for which all edge weights in the controller graph are equal to zero.




Additional information about our computational experiments, along with {\sc Matlab} source codes, can be found at:
    \begin{center}
    \href{http://www.ece.umn.edu/~mihailo/software/graphsp/index.html}{\sf www.ece.umn.edu/$\sim$mihailo/software/graphsp/}
    \end{center}

     \begin{table*}
    \centering
    \caption{Comparison of algorithms (solve times in seconds/number of iterations) for the problem of growing connected resistive Erd\"os-R\'enyi networks with different number of nodes $n$, edge probability $1.05 \log (n)/n$, and $\gamma = 0.8\,\gamma_{\max}$.}
    \label{tab.Erd}
    \begin{tabular}{|c|c|c|c|c|c|}
      \hline
      number of nodes  &  $n = 300$ & $n = 700$ & $n = 1000$ & $n = 1300$ & $n = 1500$ \\
      \hline
      number of edges  &  $m = 43986$ & $m =  242249$ & $m = 495879$ & ${m = 839487}$ & $m =  1118541$\\
      \hline\hline
      IP (PCG) & $   16.499 /8 $ & $ 394.256/ 13 $ & $   1014.282/ 13 $ & $ 15948.164/13 $ & $ 179352.208/14 $ \\
      \hline
      proxBB & $ 1.279/11 $ & $ 15.353 / 11 $  & $55.944/ 13 $ & $ 157.305/16 $ & $  239.567/16 $ \\
      \hline
      proxN  &  $  1.078/4$  & $ 11.992/4 $ & $ 34.759/4 $ & $82.488/4$ & $ 124.307/4$ \\
      \hline
    \end{tabular}
    \end{table*}
	
	\vspace*{-3ex}
\subsection{Performance comparison}
	\label{sec.pc}

In what follows, the incidence matrix of the controller graph is selected to satisfy the following requirements: (i) in the absence of the sparsity-promoting term, the closed-loop network is given by a complete graph; and (ii) there are no joint edges between the plant and the controller graphs.

%

{We first solve the problem~\eqref{eq.P1} for growing connected resistive Erd\"os-R\'enyi networks with different number of nodes. The generator of the plant dynamics is given by an undirected unweighted graph with edge probability $1.05 \log (n)/n$}. Table~\ref{tab.Erd} compares our customized algorithms in terms of speed and the number of iterations. Even for small networks, proximal methods are significantly faster than the IP method and proxN takes smaller number of iterations and converges quicker than proxBB. For a larger network (with $1500$ nodes and $1118541$ edges in the controller graph), it takes about $50$ hours for the PCG-based IP method to solve the problem. In contrast, proxN and proxBB converge in about $2$ and $4$ minutes, respectively.

Figure~\ref{fig.greedy} compares our proximal gradient algorithm with the fast greedy algorithm of~\cite{sumshalygdor15}. We solve problem~\eqref{eq.P1} for Erd\"os-R\'enyi networks with different number of nodes ($n=5$ to $500$) and $\gamma = 0.4 \,\gamma_{\max}$. After proxBB identifies the edges in the controller graph, we use the greedy method to select the same number of edges. Finally, we polish the identified edge weights for both methods. Figure~\ref{fig.greedy1} shows the solve times (in seconds) versus the number of nodes. As the number of nodes increases the proximal algorithm significantly outperforms the fast greedy method. Relative to the optimal centralized controller, both methods yield similar performance degradation of the closed-loop network; see Fig.~\ref{fig.greedy2}.

\begin{figure}
\captionsetup[subfigure]{position=top}
  \centering
\begin{tabular}{cc}
     \subfloat[solve times]{\label{fig.greedy1} \includegraphics[width=0.2\textwidth]{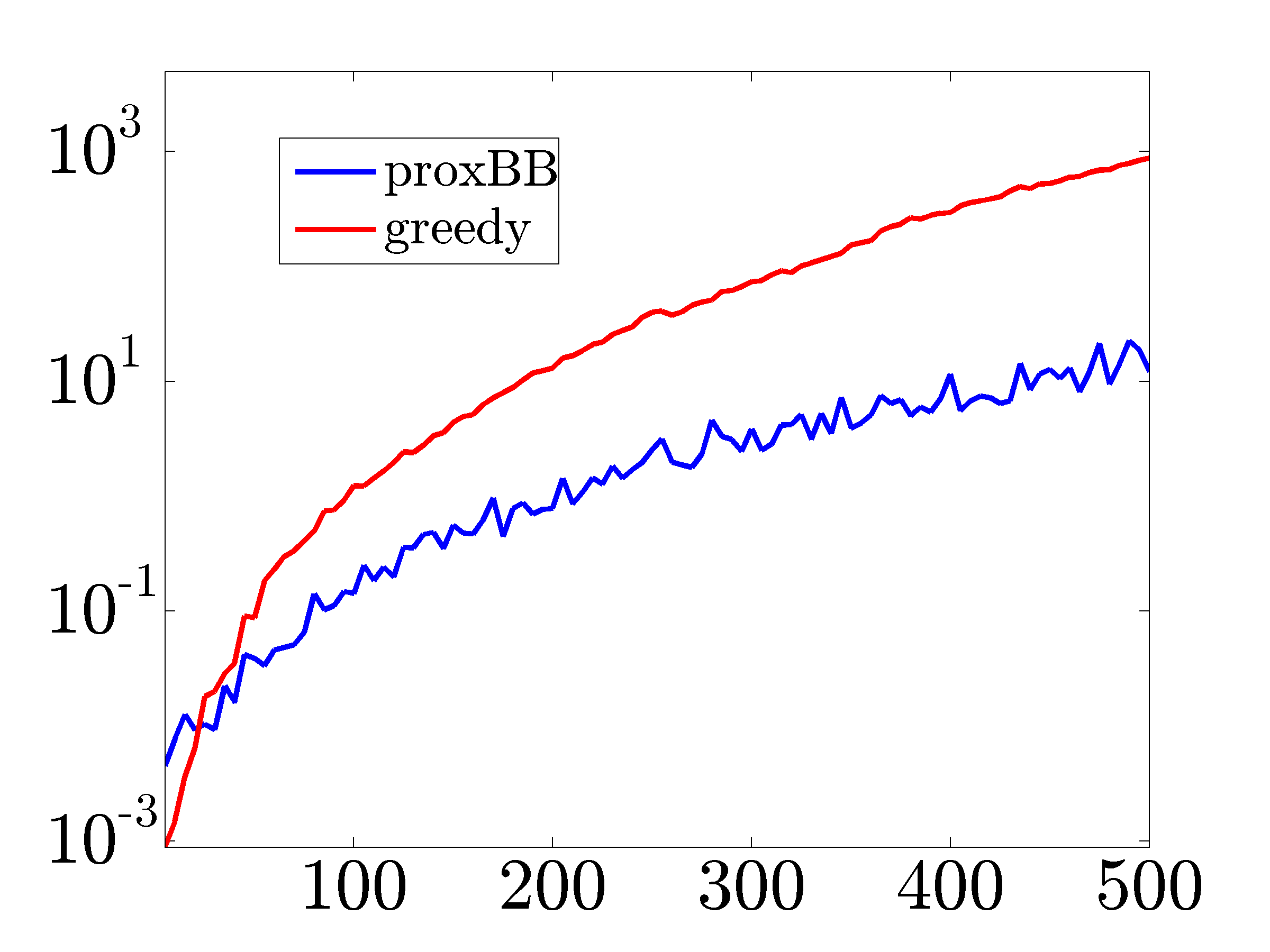}}
  &
  \subfloat[$(J-J_c)/J_c$]{\label{fig.greedy2}
     \includegraphics[width=0.2\textwidth]{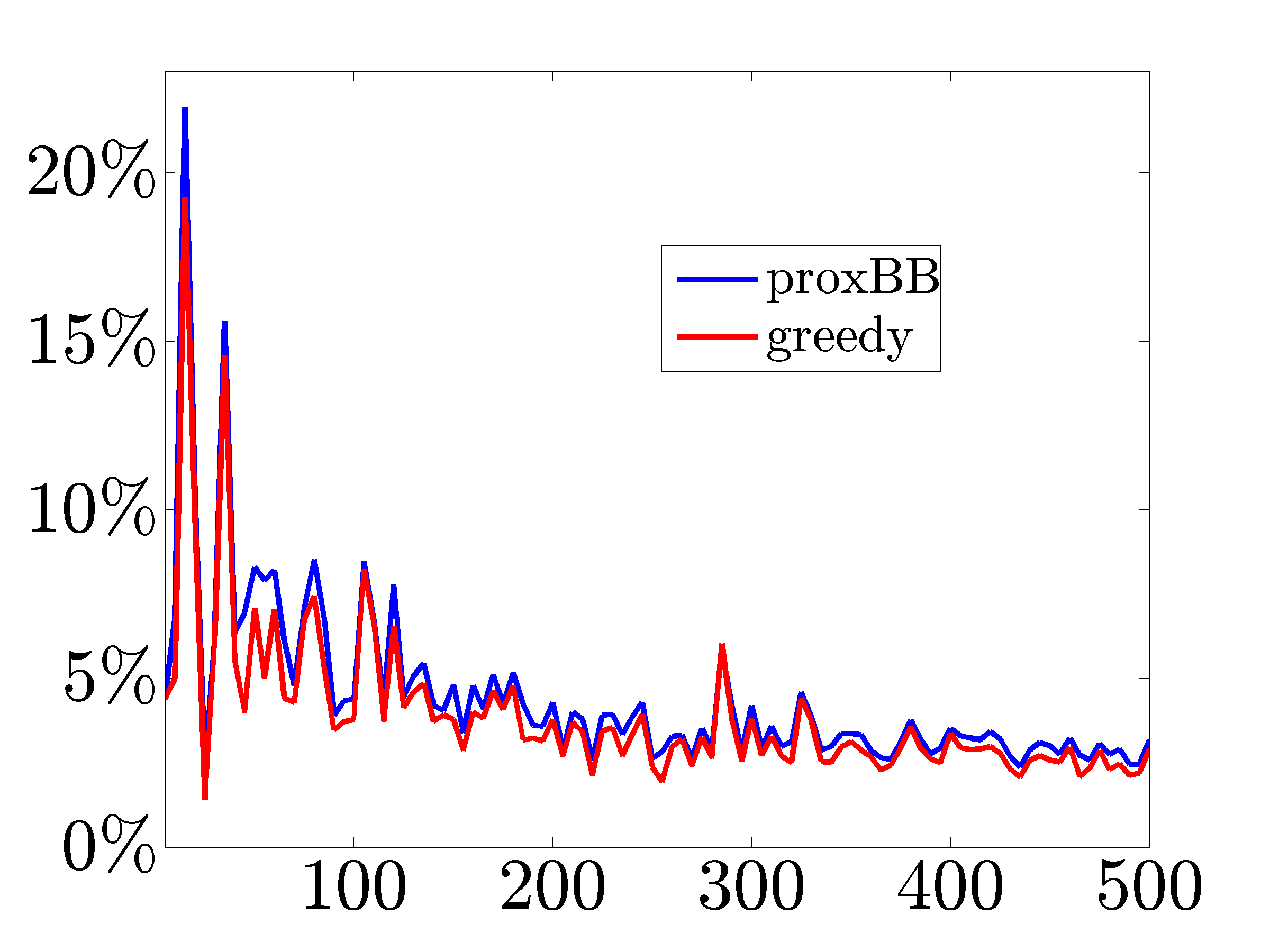}} \\[.45cm]
   $n$ & $n$
\end{tabular}
 \caption{(a) Solve times (in seconds); and (b) performance degradation (in percents) of proximal gradient and greedy algorithms relative to the optimal centralized controller.}
      \label{fig.greedy}
\end{figure}

\vspace*{-2ex}
\subsection{Large-scale Facebook network}

To evaluate effectiveness of our algorithms on large networks, we solve the problem of growing a network of friendships. In such social networks, nodes denote people and edges denote friendships. There is an edge between two nodes if two people are friends. {The network is obtained by examining social network of 10 users (the so-called ego nodes); all other nodes are friends to at least one of these ego nodes~\cite{mcales12}.} The resulting network is undirected and unweighted with $4039$ nodes and $88234$ edges; the data is available at \href{http://snap.stanford.edu/data/}{http://snap.stanford.edu/data/}.
Our objective is to improve performance by adding a small number of extra edges. We assume that people can only form friendships with friends of their friends. This restricts the number of potential edges in the controller graph to $1358067$.

To avoid memory issues, we have implemented our algorithms in C++. For $\gamma = c \, \gamma_{\max}$ with $c = \{0.1, 0.2, 0.5, 0.8 \}$ and $\gamma_{\max} = 19.525$, the proximal gradient algorithm computes the solution in about $10$, $2.6$, $0.87$, and $0.43$ hours, respectively. After designing the topology of the controller graph, we optimize the resulting edge weights via polishing.

Figure~\ref{fig.nnzvsgammafb} shows that the number of nonzero elements in the vector $x$ decreases as $\gamma$ increases and Fig.~\ref{fig.jvsnnzfb} illustrates that the $\htwo$ performance deteriorates as the number of nonzero elements in $x$ decreases. In particular, for $\gamma = 0.8\, \gamma_{\max}$, the identified sparse controller has only $3$ nonzero elements (it uses only $0.0002\%$ of the potential edges). Relative to the optimal centralized controller, this controller degrades performance by $16.842\%$,
    $
    (J - J_c)/J_c = 16.842\%.
    $

\begin{figure}
\captionsetup[subfigure]{position=top}
  \centering
\begin{tabular}{cc}
       \subfloat[$\card (x)$]{\label{fig.nnzvsgammafb} \includegraphics[width=0.2\textwidth]{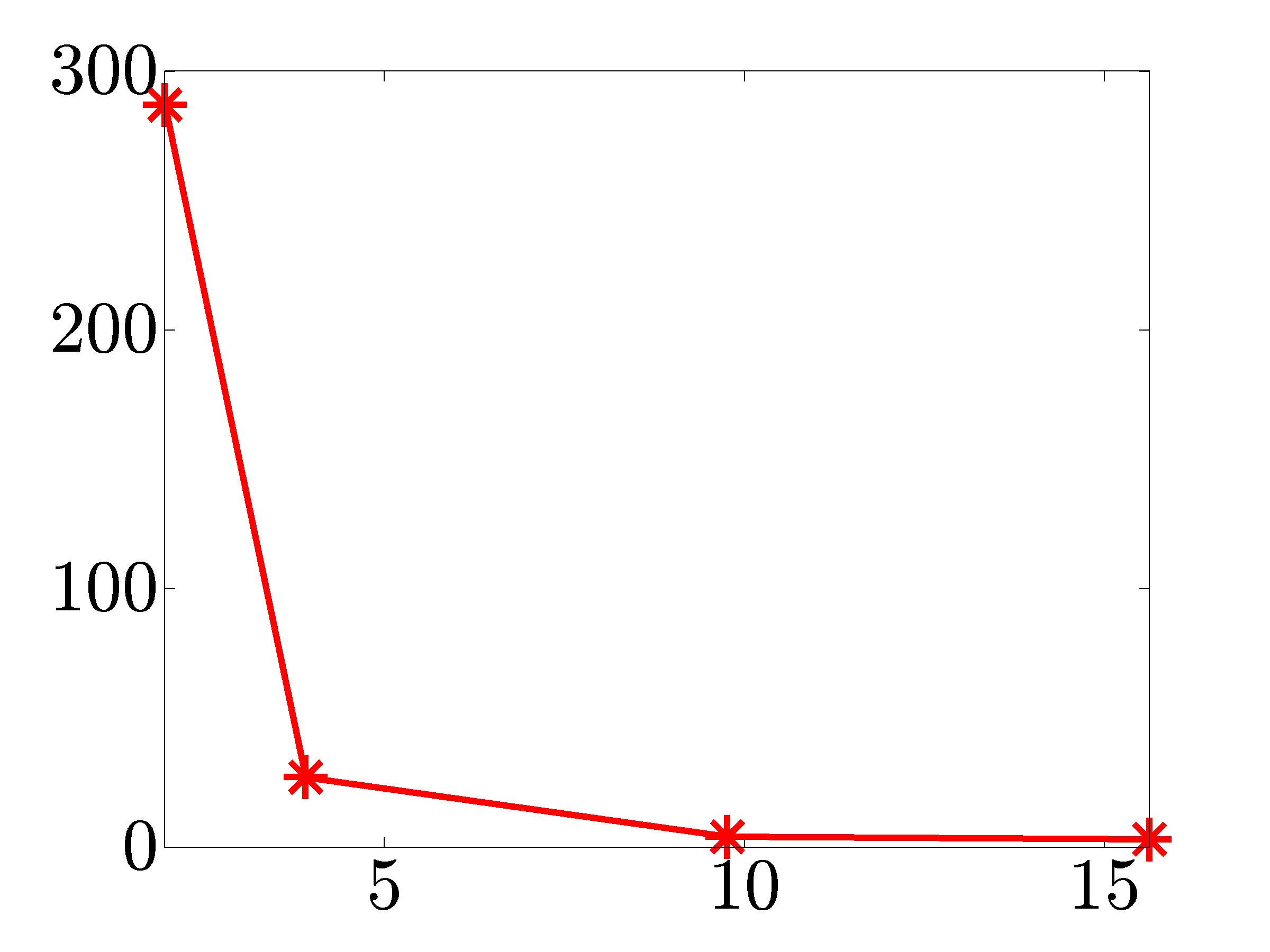}}
  &
  \subfloat[$(J-J_c)/J_c$]{\label{fig.jvsnnzfb}
     \includegraphics[width=0.19\textwidth]{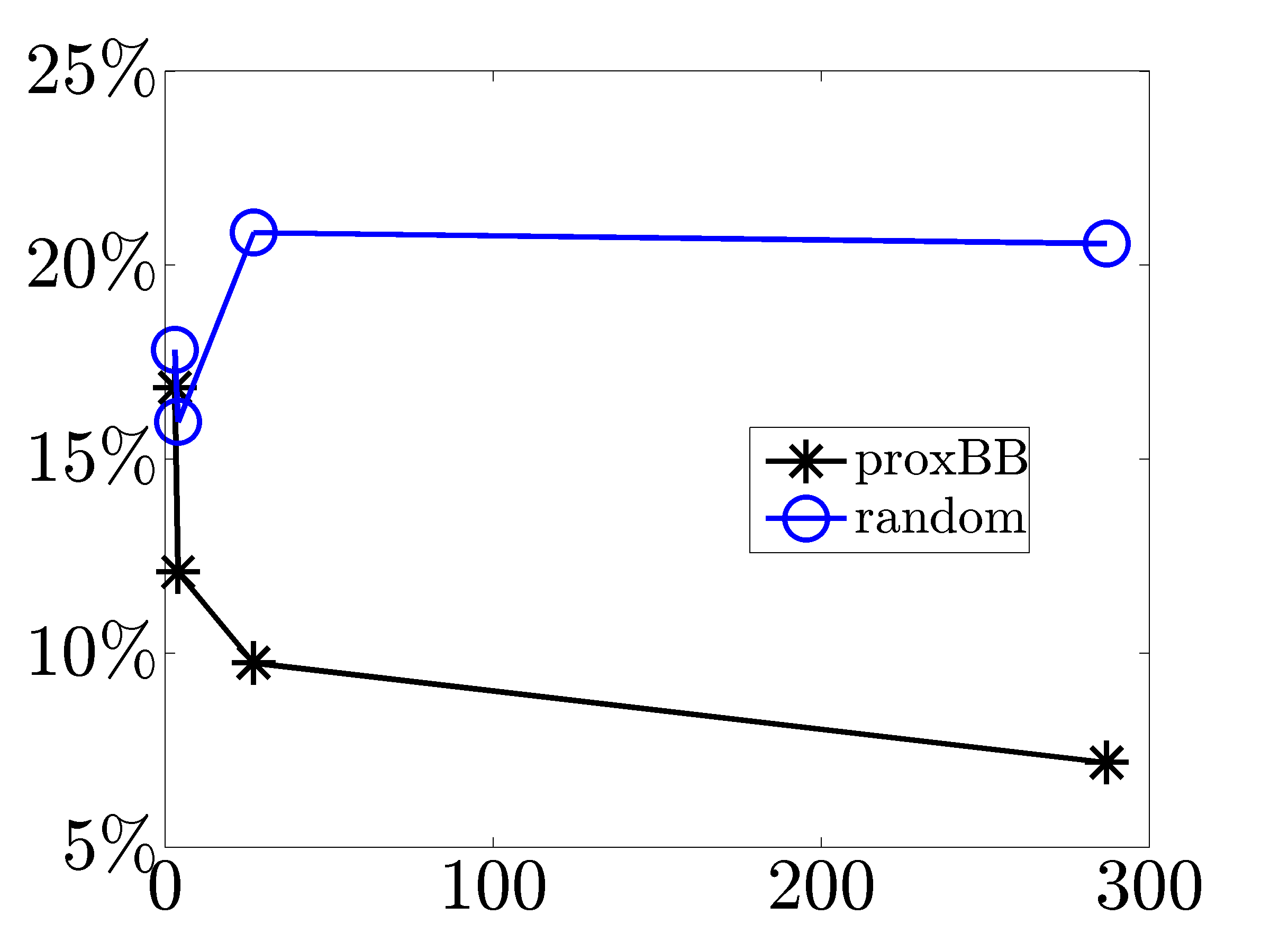}}
     \\
     $\gamma$ &  $\card (x)$
\end{tabular}
 \caption{(a) Sparsity level; and (b) optimal tradeoff curves resulting from the application of proximal gradient algorithm and a heuristic strategy for the Facebook network.}
      \label{fig.fb}
\end{figure}

In all of our experiments, the added links with the largest edge weights connect either the ego nodes to each other or three non-ego nodes to the ego nodes. Thus, our method {recognizes significance of the ego nodes and identifies non-ego nodes that play an important role in improving performance.}

We compare performance of the identified controller to a heuristic strategy that is described next. {The} controller graph contains $16$ potential edges between ego nodes. If the number of edges identified by our method is smaller than $16$, we randomly select the desired number of edges between ego nodes. Otherwise, we connect all ego nodes and select the remaining edges in the controller graph randomly. We then use polishing to find the optimal edge weights. The performance of resulting random controller graphs are averaged over $10$ trials and the performance loss relative to the optimal centralized controller is displayed in Fig.~\ref{fig.jvsnnzfb}. We see that our algorithm always performs better than the heuristic strategy. On the other hand, the heuristic strategy outperforms the strategy that adds edges randomly (without paying attention to ego nodes). Unlike our method, the heuristic strategy does not necessarily improve the performance by increasing the number of added edges. In fact, the performance deteriorates as the number of edges in the controller \mbox{graph increases from $4$ to $27$; see Fig.~\ref{fig.jvsnnzfb}.}

\vspace*{-2ex}
\subsection{Random disconnected network}


The plant graph (blue lines) in Fig.~\ref{fig.disconnected} contains $50$ randomly distributed nodes in a region of $10 \times 10$ units. Two nodes are neighbors if their Euclidean distance is not greater than $2$ units. We examine the problem of adding edges to a plant graph which is not connected and solve the sparsity-promoting optimal control problem~\eqref{eq.SP} for controller graph with $m = 1094$ potential edges. This is done for $200$ logarithmically-spaced values of $\gamma \in [10^{-3}, \, 2.5]$ using the path-following iterative reweighted algorithm as a proxy for inducing sparsity~\cite{canwakboy08}. As indicated by~\eqref{eq.weights}, we set the weights to be inversely proportional to the magnitude of the solution $x$ to~\eqref{eq.SP} at the previous value of $\gamma$. We choose $\varepsilon = 10^{-3}$ in~\eqref{eq.weights} and initialize weights for $\gamma = 10^{-3}$ using the solution to~\eqref{eq.SP} with $\gamma= 0$ (i.e., the optimal centralized vector of the edge weights). Topology design is followed by the polishing step that computes the optimal edge weights; see Section~\ref{sec.polish}.


As illustrated in Fig.~\ref{fig.disconnected}, larger values of $\gamma$ yield sparser controller graphs (red lines). In contrast to all other examples, the plant graph is not connected and the optimal solution is obtained using the algorithms of Section~\ref{sec.IPgen}. Note that greedy method~\cite{sumshalygdor15} cannot be used here. Since the plant graph has three disconnected subgraphs, at least two edges in the controller are needed to make the \mbox{closed-loop network connected.}

   	\begin{figure}
      \centering
      \begin{tabular}{c c}
      \subfloat[$\gamma =  0.02$]
      {\label{fig.network_local}
      \includegraphics[width=0.2\textwidth]{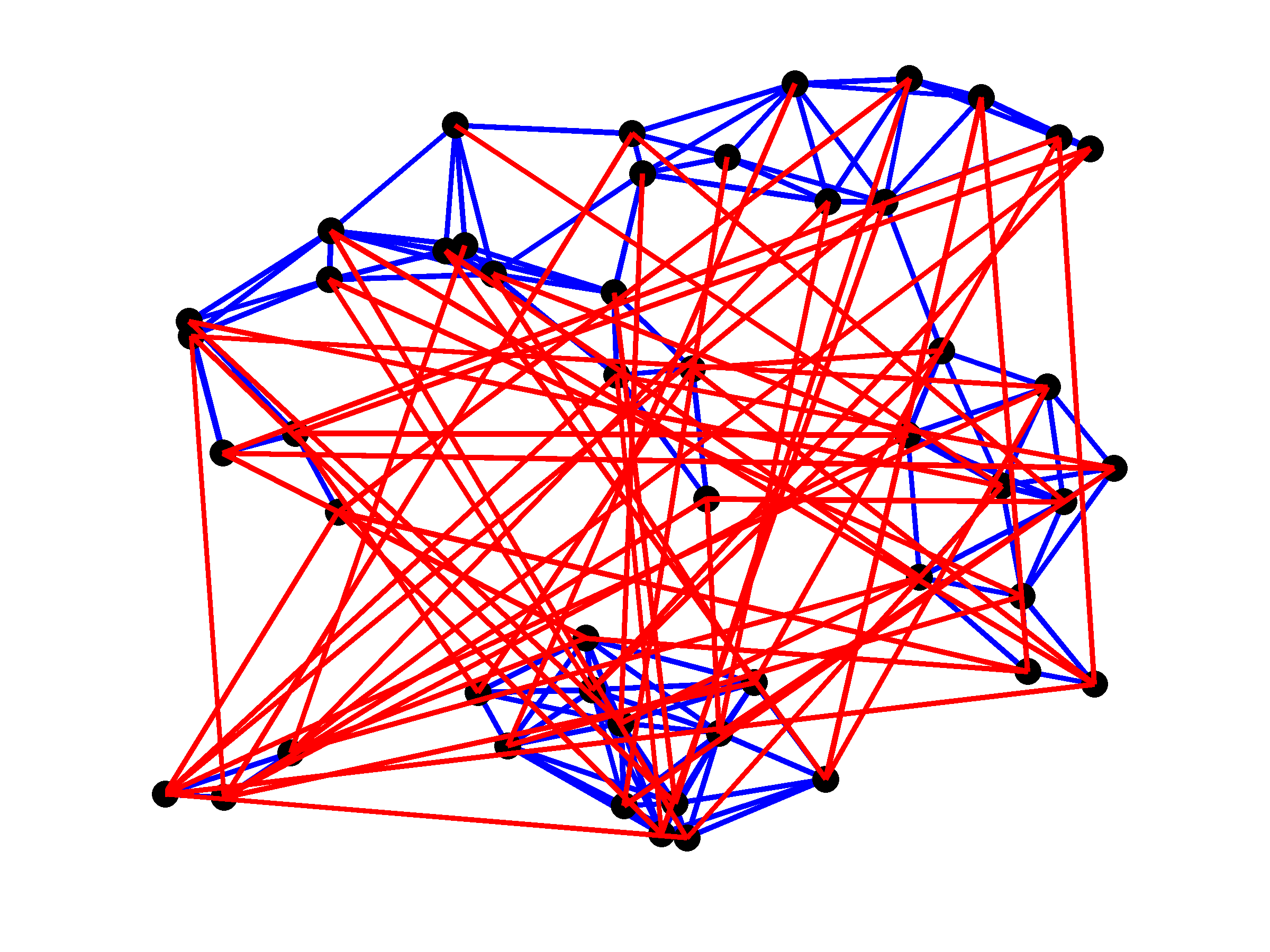}
      }
      &
      \subfloat[$\gamma =  0.09$]
      {\label{fig.network_local}
      \includegraphics[width=0.2\textwidth]{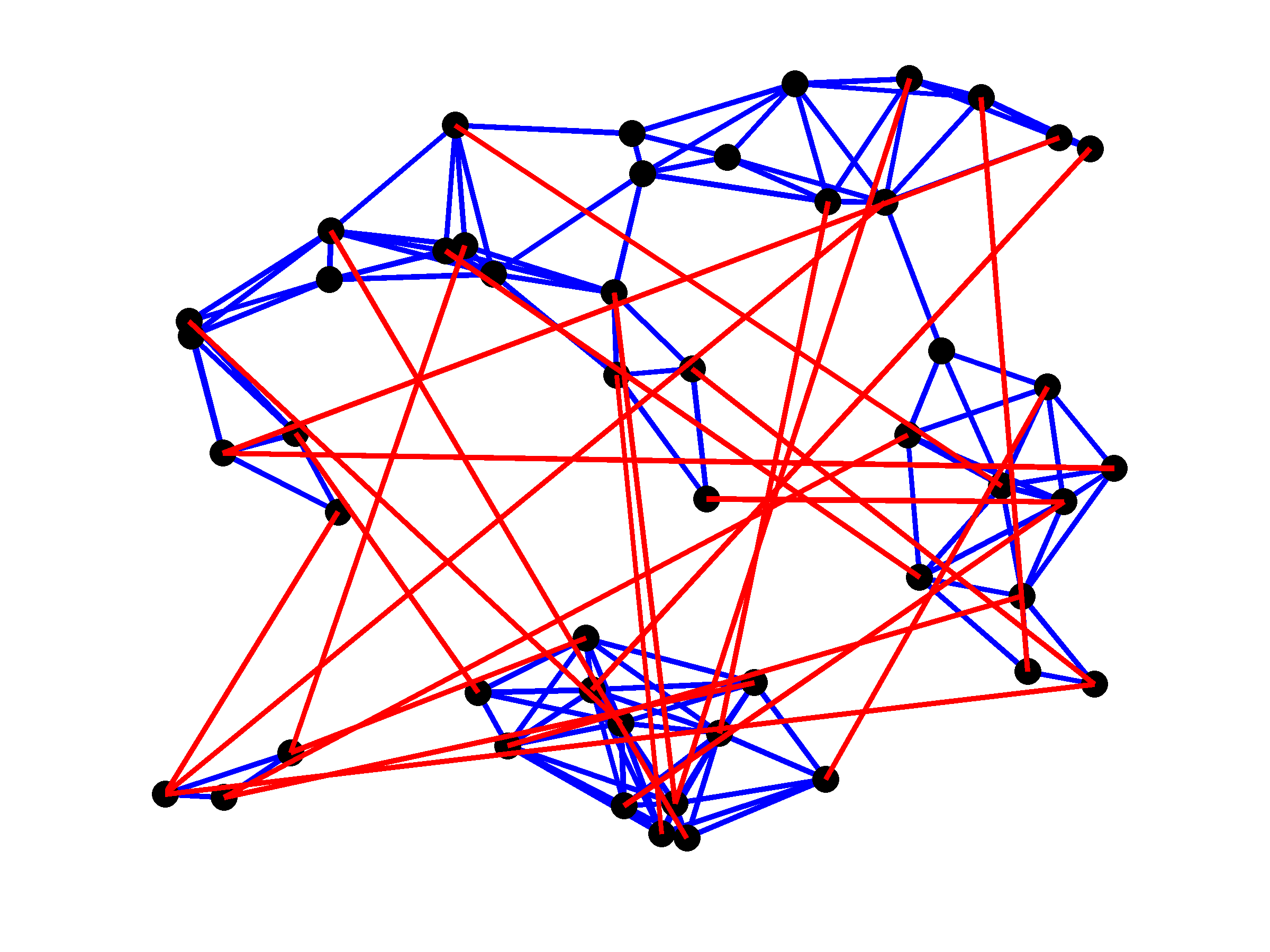}
      }
      \end{tabular}
      \\[-0.35cm]
      \begin{tabular}{c c}
      \subfloat[$\gamma =  0.63$]
      {\label{fig.network_longrange}
      \includegraphics[width=0.2\textwidth]{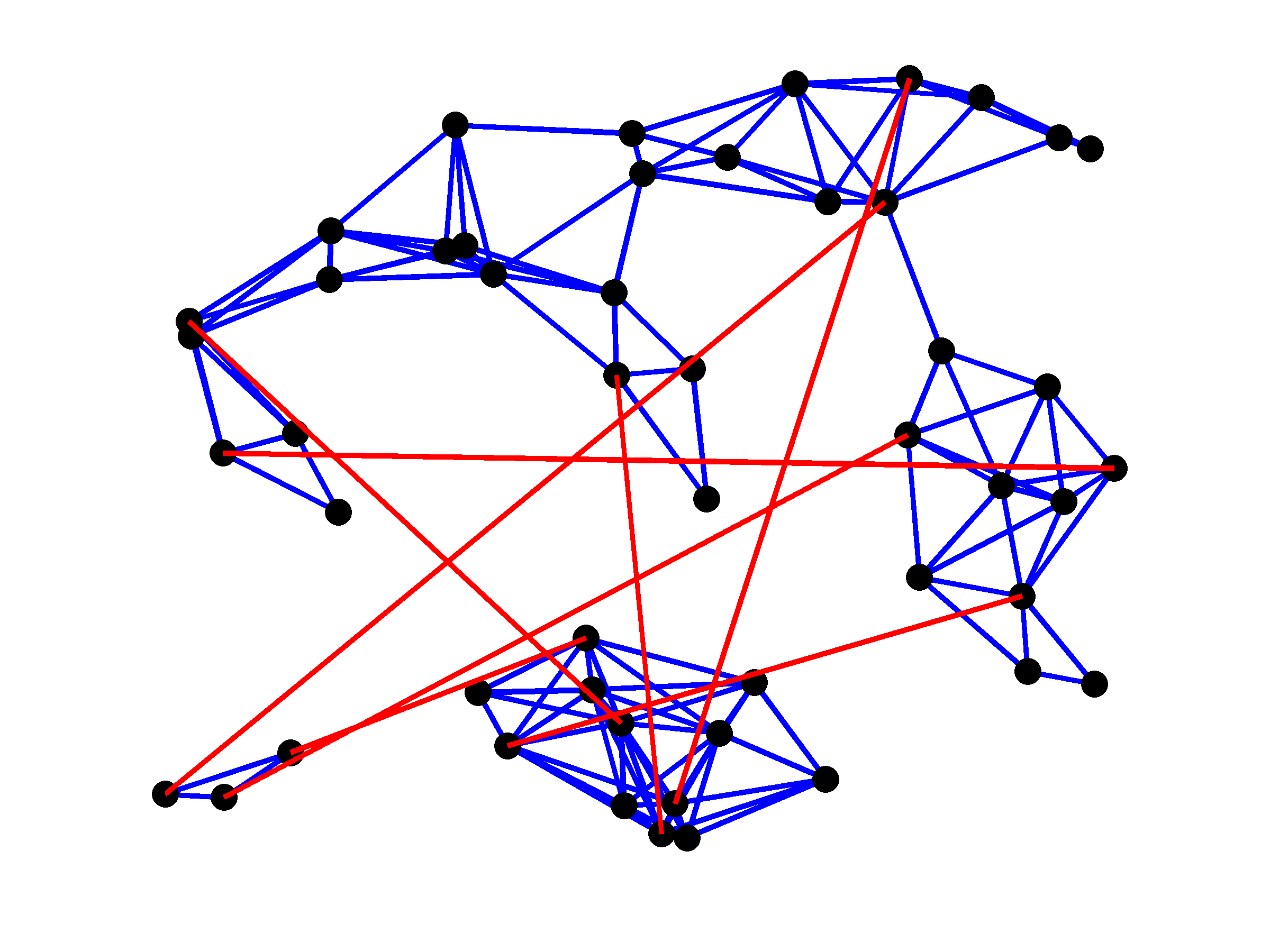}
      }
      &
      \subfloat[$\gamma =  2.5$]
      {\label{fig.network_local}
      \includegraphics[width=0.2\textwidth]{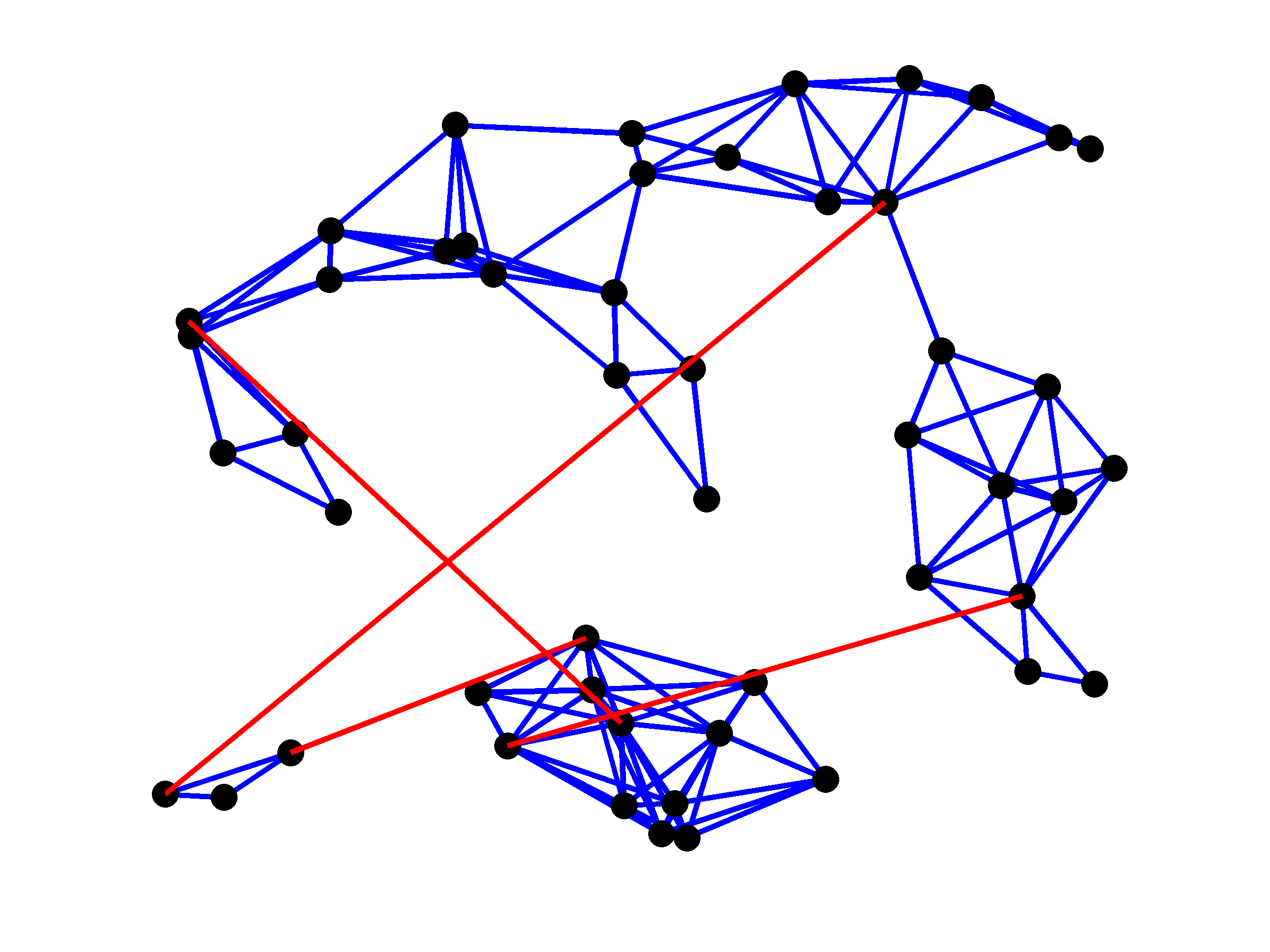}
      }
      \end{tabular}
      \caption{Topologies of the plant (blue lines) and controller graphs (red lines) for an unweighted random network with three disconnected subgraphs.}
      \label{fig.disconnected}
    \end{figure}

Figure~\ref{fig.Hpath} shows that the number of nonzero elements in the vector of the edge weights $x$ decreases and that the closed-loop performance deteriorates as $\gamma$ increases. In particular, Fig.~\ref{fig.JvsNNZ} illustrates the optimal tradeoff curve between the $\htwo$ performance loss (relative to the optimal centralized controller) and the sparsity of the vector $x$. For $\gamma = 2.5$, only four edges are added. Relative to the optimal centralized vector of the controller edge weights $x_c$, the identified sparse controller in this case uses only $0.37\%$ of the edges, and achieves a performance loss of $82.13\%$, i.e.,
    $
    \card (x)/\card (x_c) = 0.37\%
    $
    and
   $
   (J - J_c)/J_c = 82.13\%.
   $
Here, $x_c$ is the solution to~\eqref{eq.SP} with $\gamma= 0$ and the pattern of non-zero elements of $x$ is obtained by solving~\eqref{eq.SP} with $\gamma =  2.5$ via the path-following iterative reweighted algorithm.
	
\vspace*{-2ex}
\subsection{Path and ring networks}

For path networks, our computational experiments show that for a large enough value of the sparsity-promoting parameter $\gamma$ a single edge, which generates the longest cycle, is added; see Fig.~\ref{fig.path}, top row. This is in agreement with~\cite{zelschall13} where it was proved that the longest cycle is most beneficial for improving the ${\cal H}_2$ performance of tree networks. Similar observations are made for the spatially-invariant ring network with nearest neighbor interactions. For large values of $\gamma$, each node establishes a link to the node that is farthest away in the network; see Fig.~\ref{fig.path}, bottom row. This is in agreement with recent theoretical developments~\cite{farzhalinjovCDC14} where perturbation analysis was used to identify optimal week links in edge-transitive consensus networks. Thus, for these regular networks and large enough values of the regularization parameter, our approach indeed provides the globally optimal solution to the original non-convex cardinality minimization problem.

	 \begin{figure*}
\captionsetup[subfigure]{position=top}
  \centering
\begin{tabular}{ccc}
 \subfloat[$\card (x)/\card (x_c)$]{\label{fig.Jvzgamma} \includegraphics[width=0.25\textwidth]{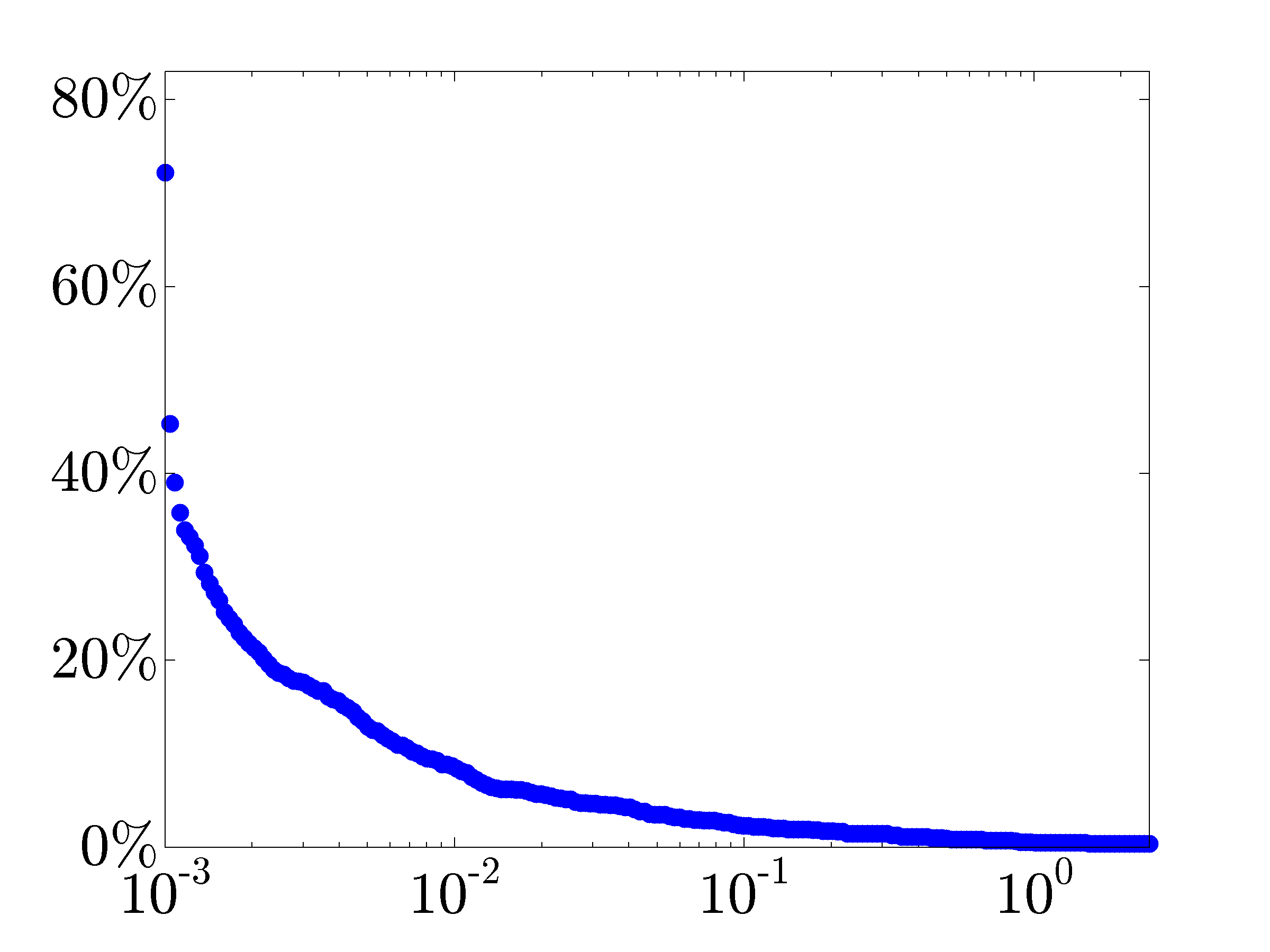}}
  &
     \subfloat[$(J-J_c)/J_c$]{\label{fig.NNZvsgamma} \includegraphics[width=0.25\textwidth]{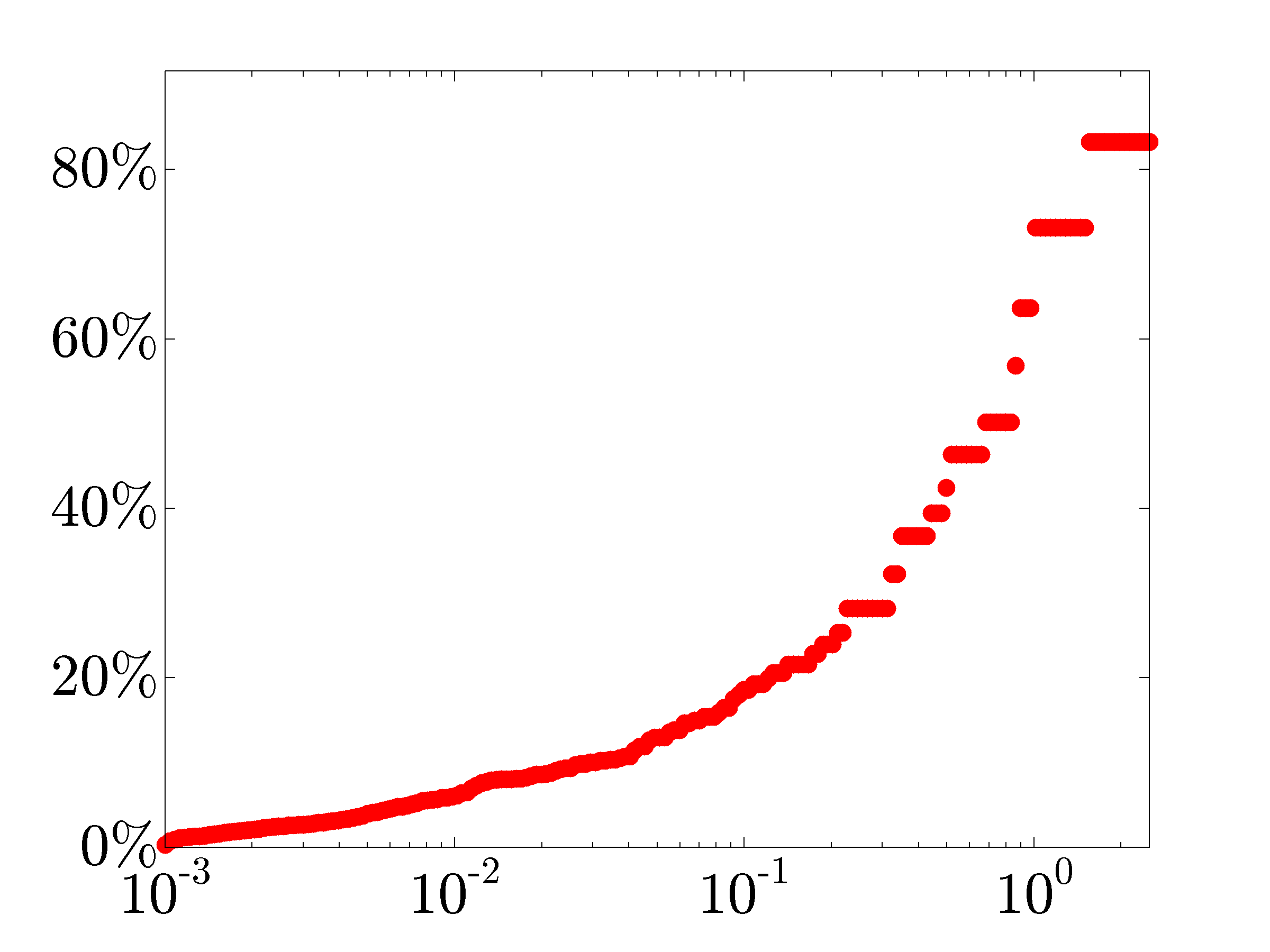}}
  &
  \subfloat[$(J-J_c)/J_c$]{\label{fig.JvsNNZ}
      \includegraphics[width=0.25\textwidth]{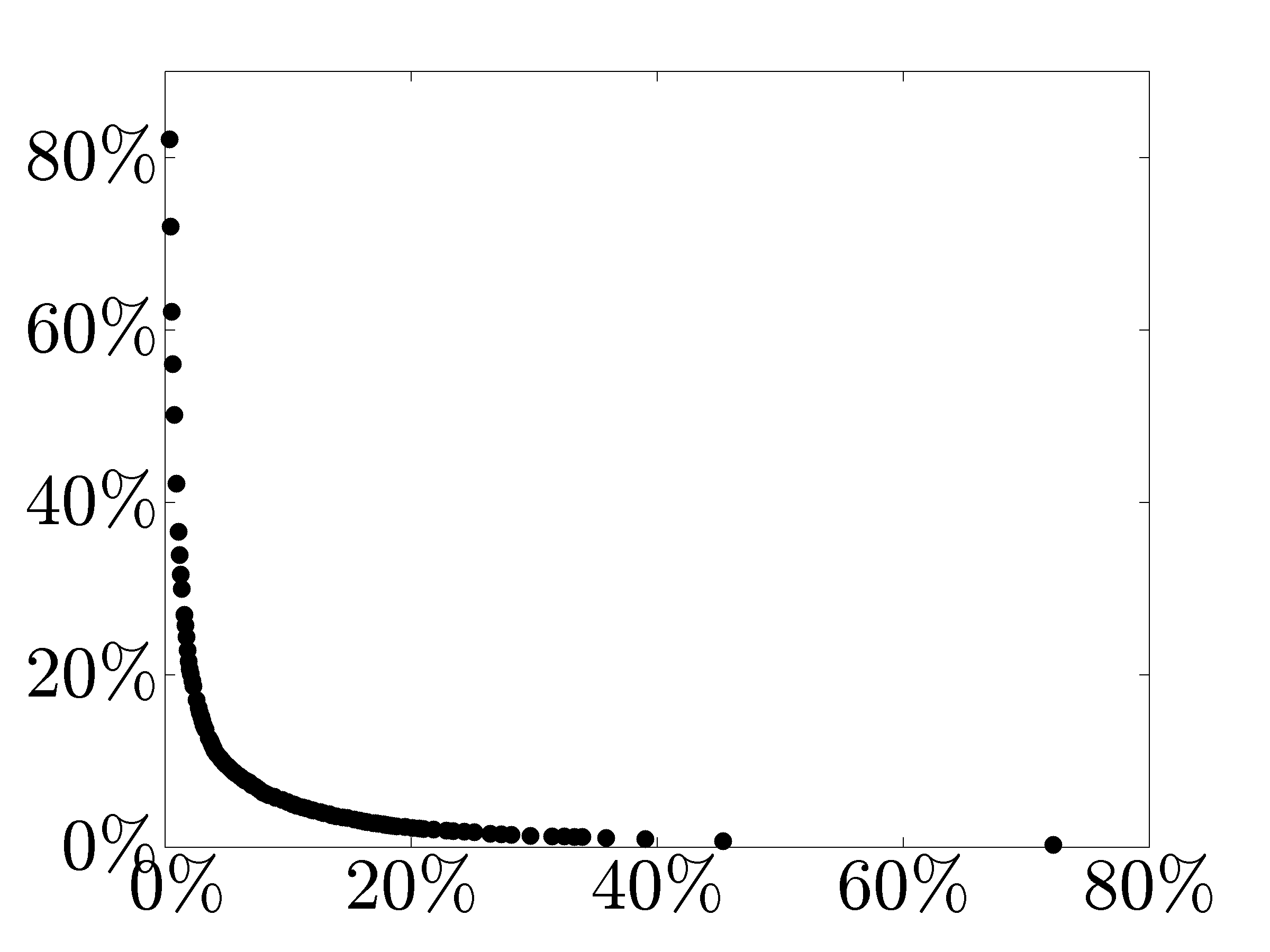}}
      \\
   $\gamma$ & $\gamma$ &  $\card (x)/\card (x_c)$
\end{tabular}
 \caption{(a) Sparsity level; (b) performance degradation; and (c) the optimal tradeoff curve between the performance degradation and the sparsity level of optimal sparse $x$ compared to the optimal centralized vector of the edge weights $x_c$. The results are obtained for unweighted random disconnected plant network with topology shown in Fig.~\ref{fig.disconnected}.}
      \label{fig.Hpath}
\end{figure*}

\begin{figure*}
      \centering
      \begin{tabular}{cccc}
      \subfloat[$\gamma  = 0 $]
      {\label{fig.path1}
      \includegraphics[width=0.2\textwidth]{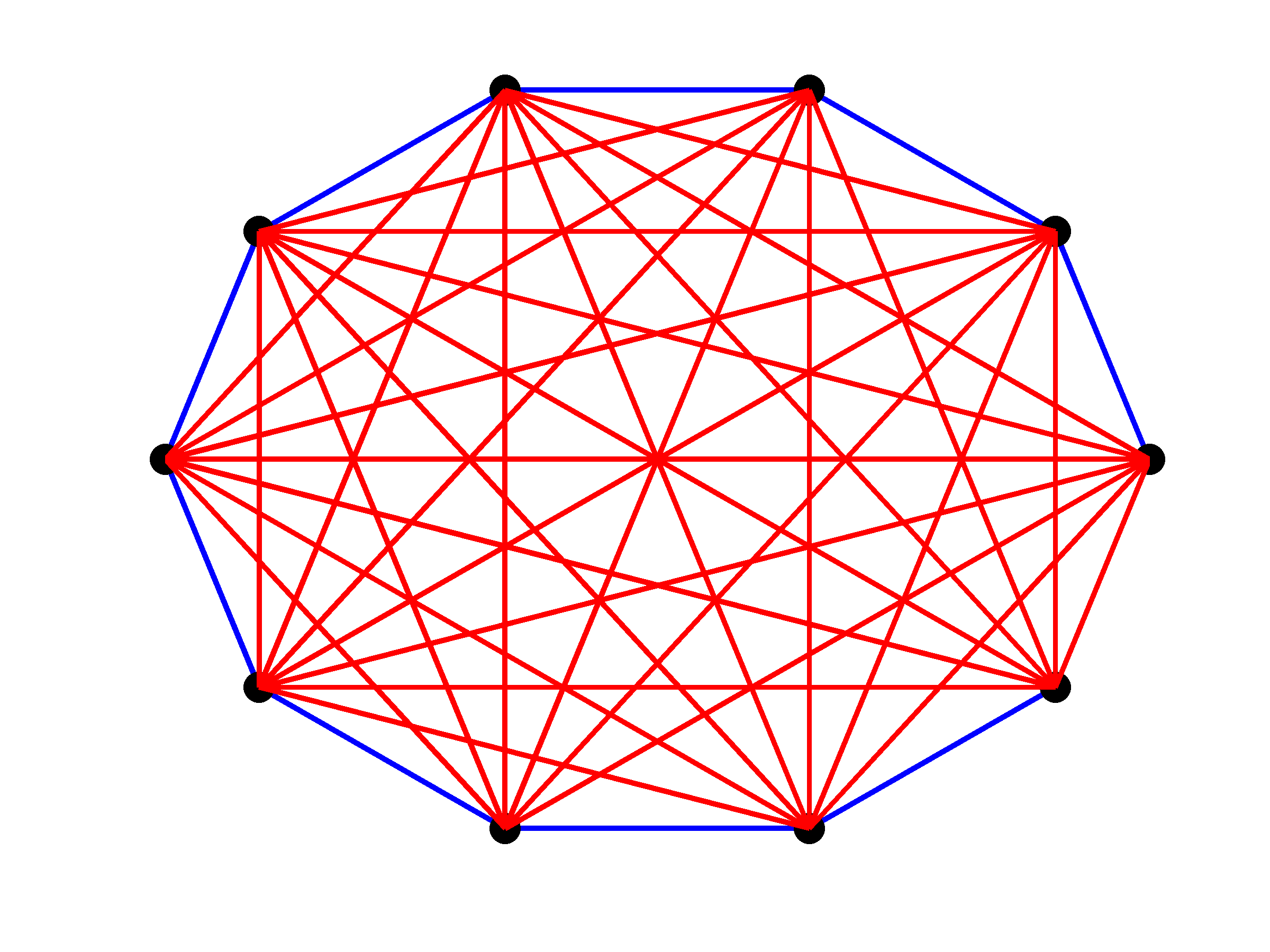}
      }
      &
      \subfloat[$\gamma  = 0.09 \,\gamma_{\max}$]
      {\label{fig.path2}
      \includegraphics[width=0.2\textwidth]{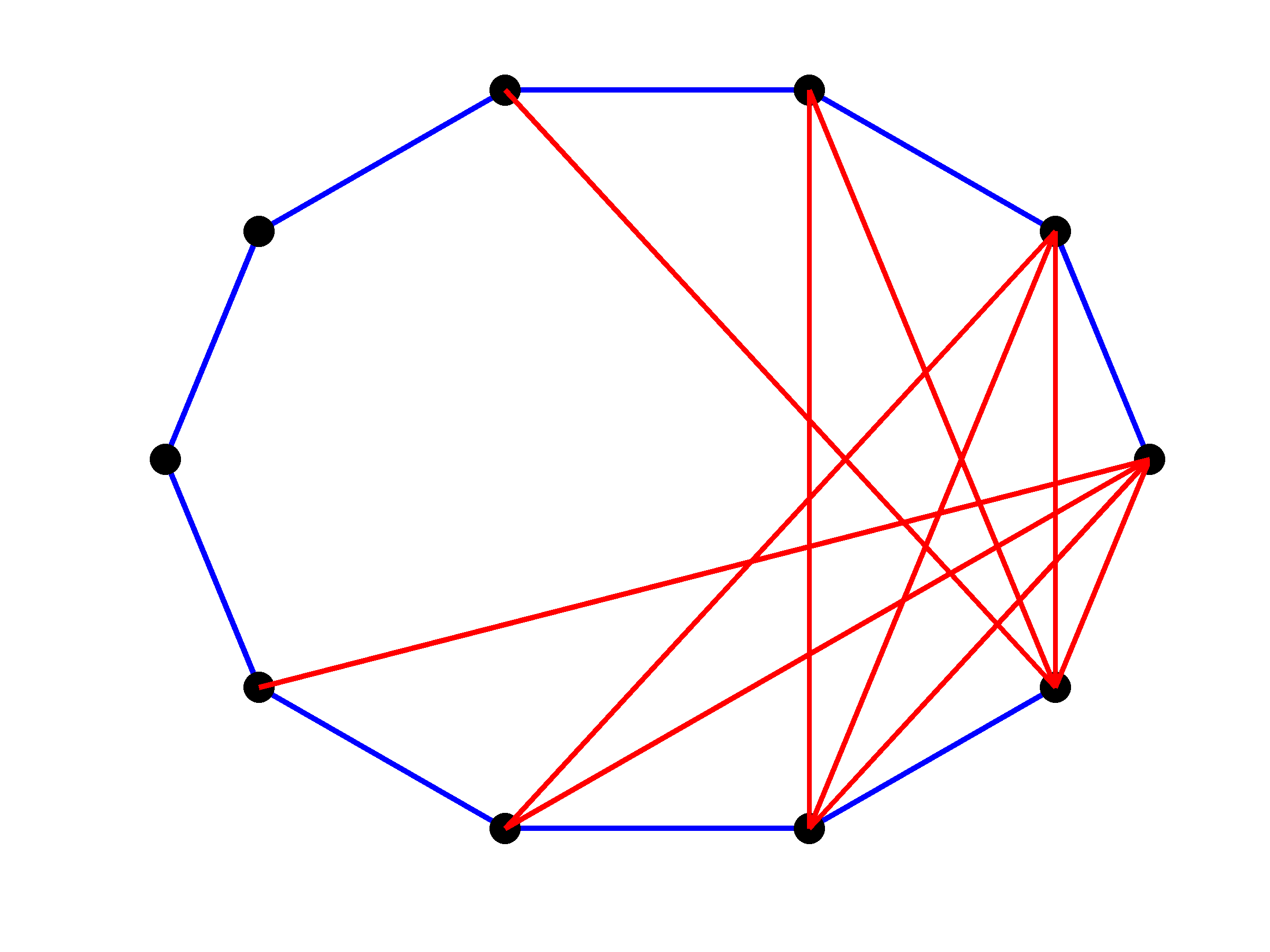}
      }
     &
      \subfloat[$\gamma  = 0.24 \,\gamma_{\max}$]
      {\label{fig.path3}
      \includegraphics[width=0.2\textwidth]{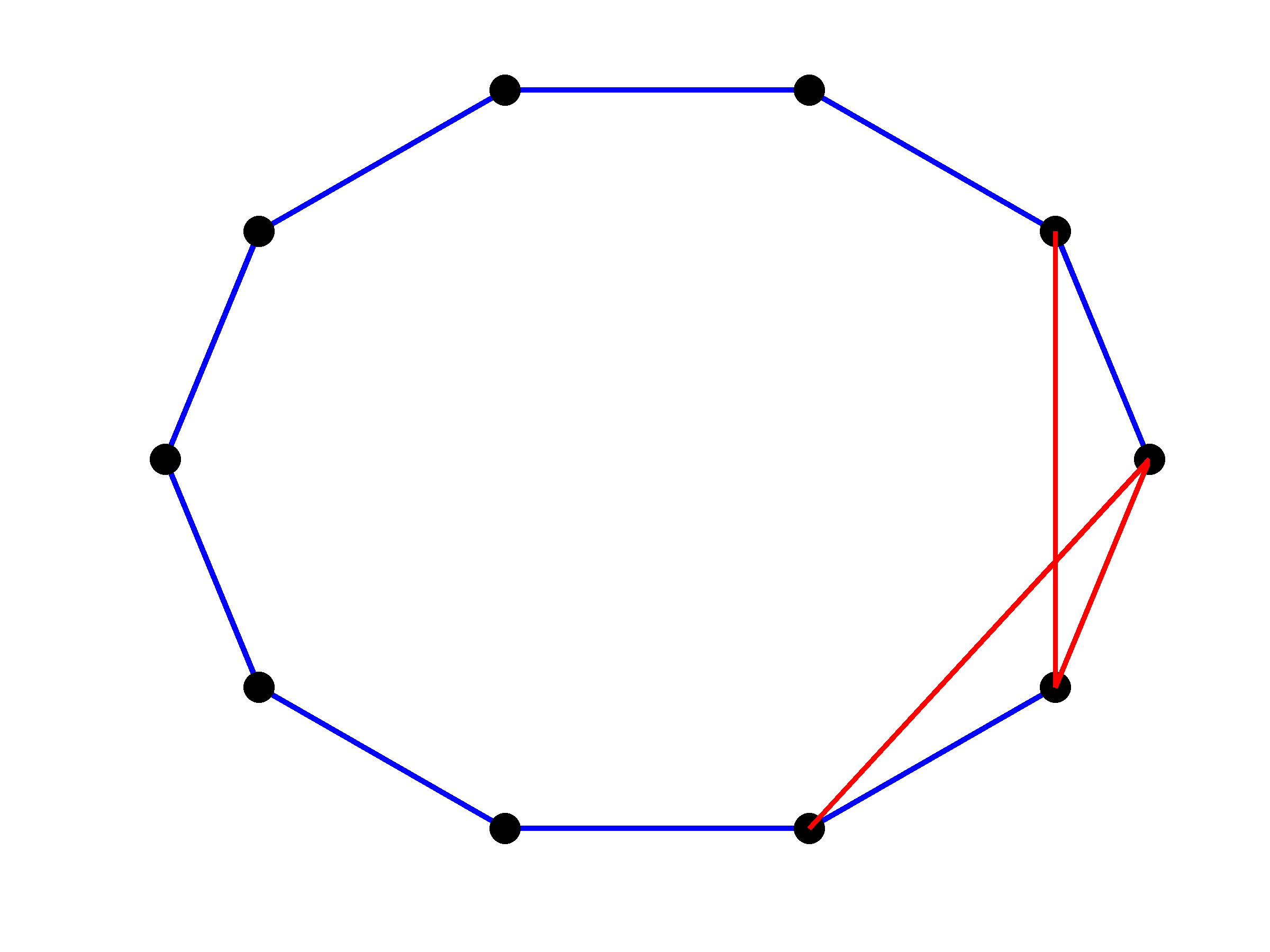}
      }
      &
      \subfloat[$\gamma  = 0.96 \,\gamma_{\max}$]
      {\label{fig.path4}
      \includegraphics[width=0.2\textwidth]{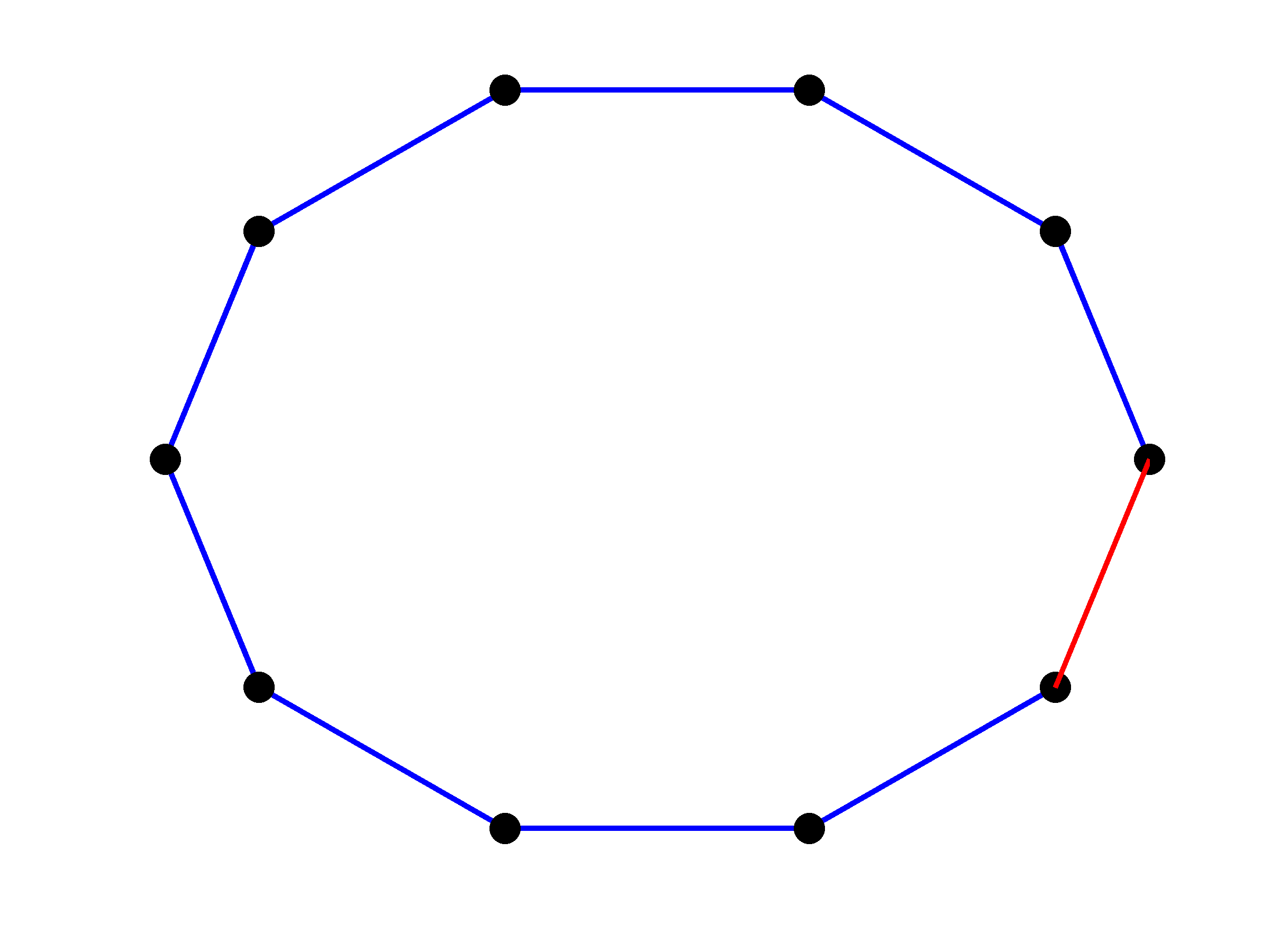}
      }
      \\
       \subfloat[$\gamma = 0$]
      {\label{fig.cycle1}
      \includegraphics[width=0.2\textwidth]{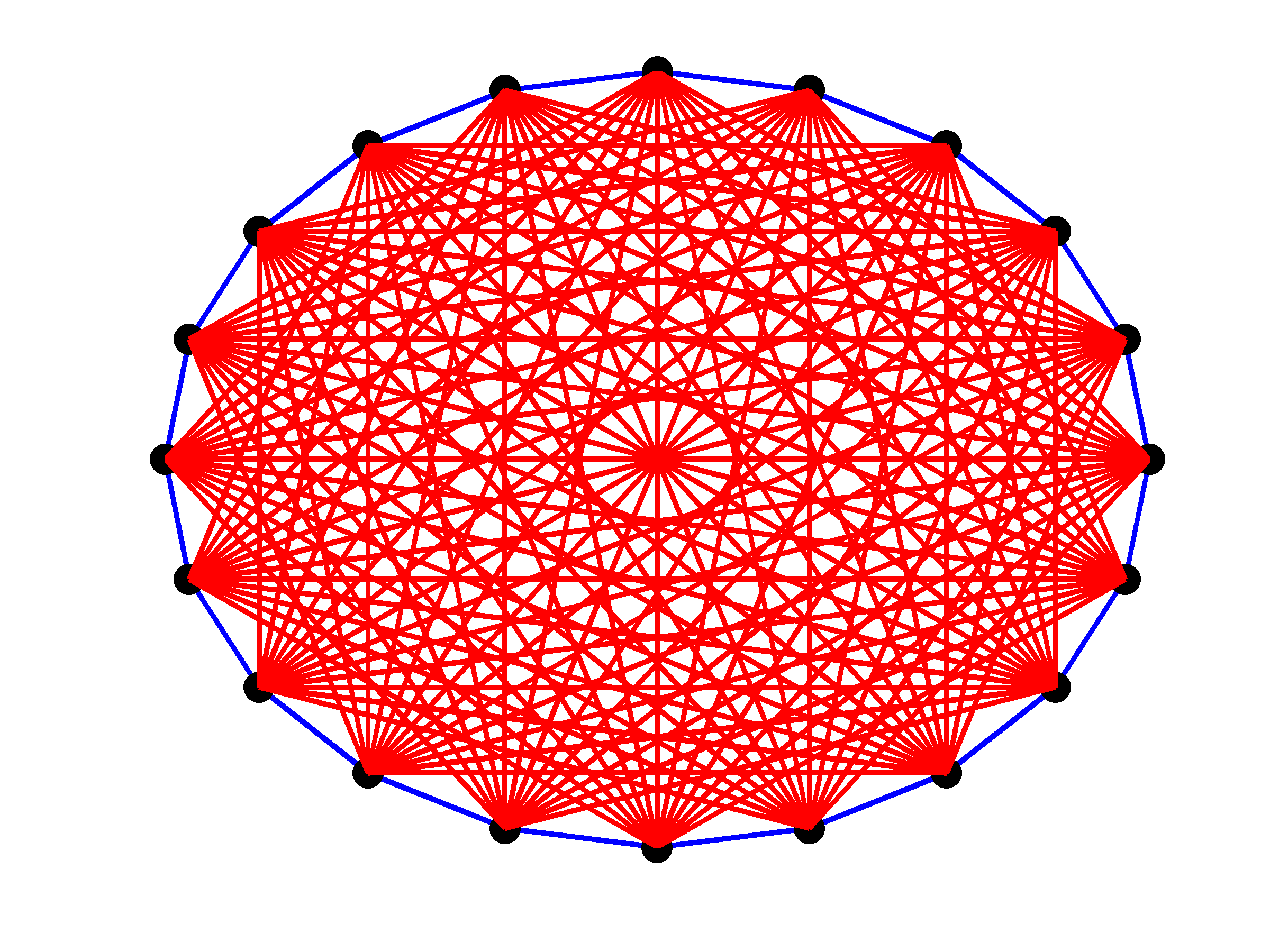}
      }
      &
      \subfloat[$\gamma = 0.11 \,\gamma_{\max}$]
      {\label{fig.cycle2}
      \includegraphics[width=0.2\textwidth]{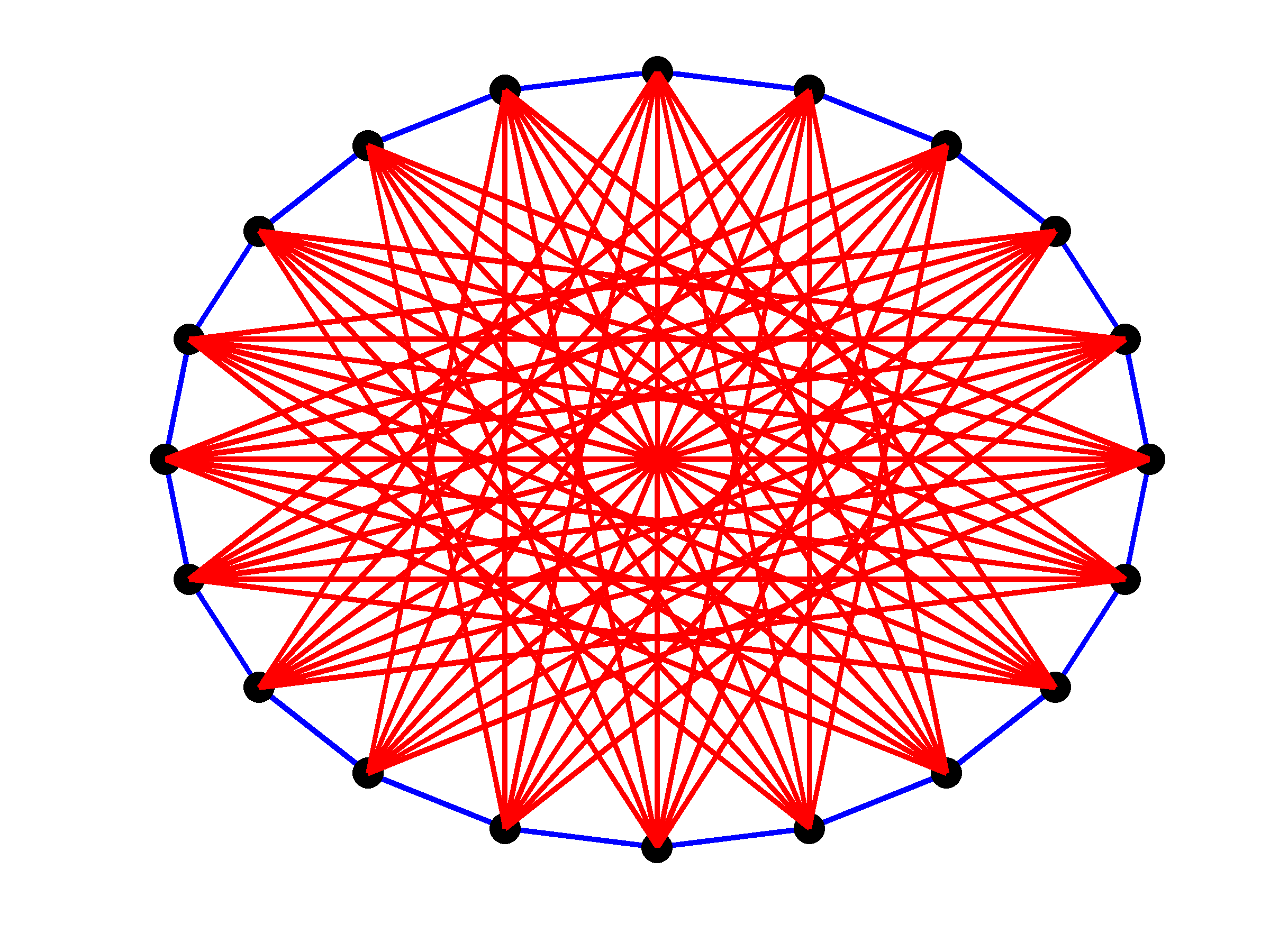}
      }
      &
      \subfloat[$\gamma  = 0.24 \,\gamma_{\max}$]
      {\label{fig.cycle3}
      \includegraphics[width=0.2\textwidth]{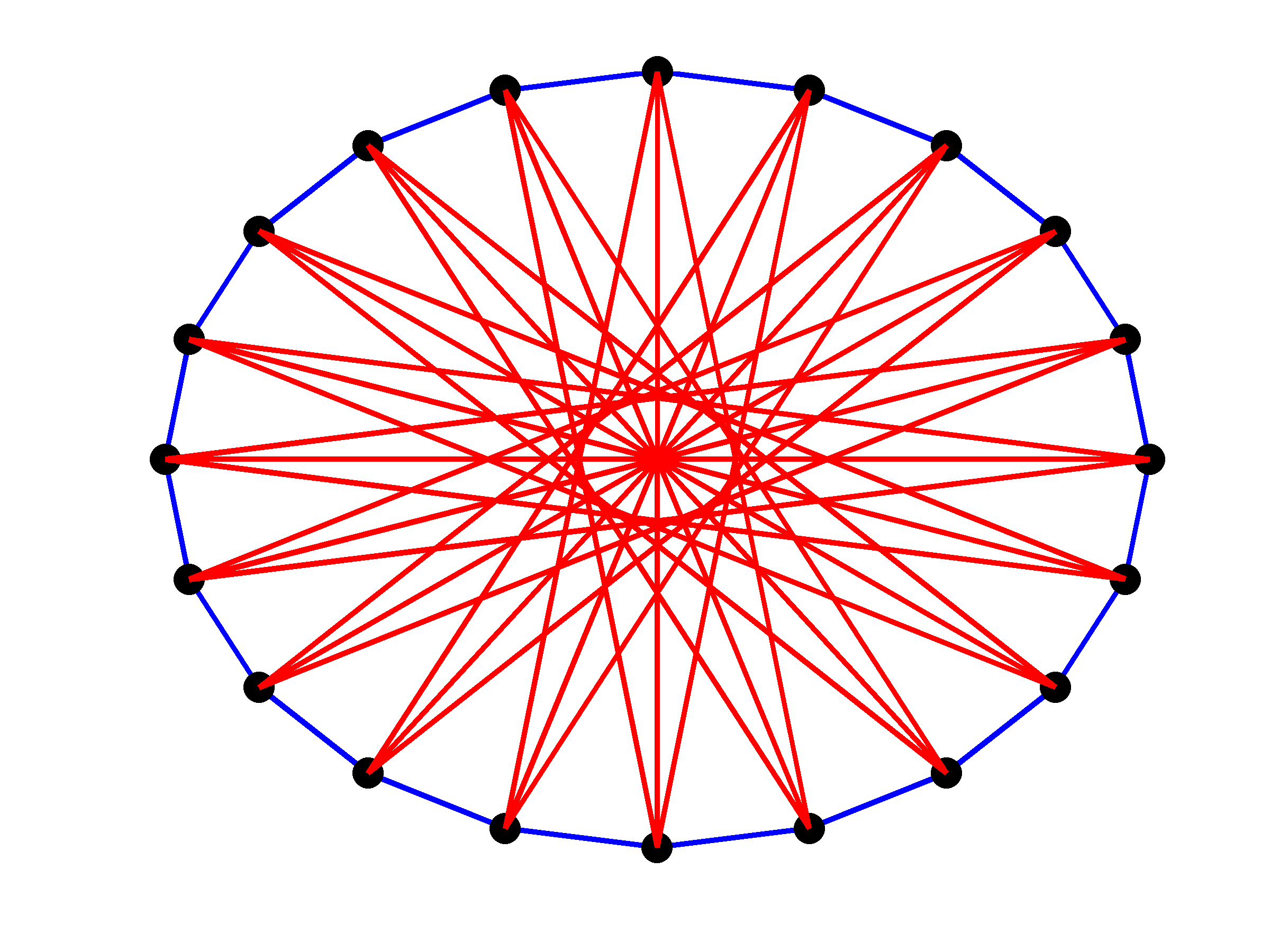}
      }
      &
      \subfloat[$\gamma  = 0.94 \,\gamma_{\max}$]
      {\label{fig.cycle4}
      \includegraphics[width=0.2\textwidth]{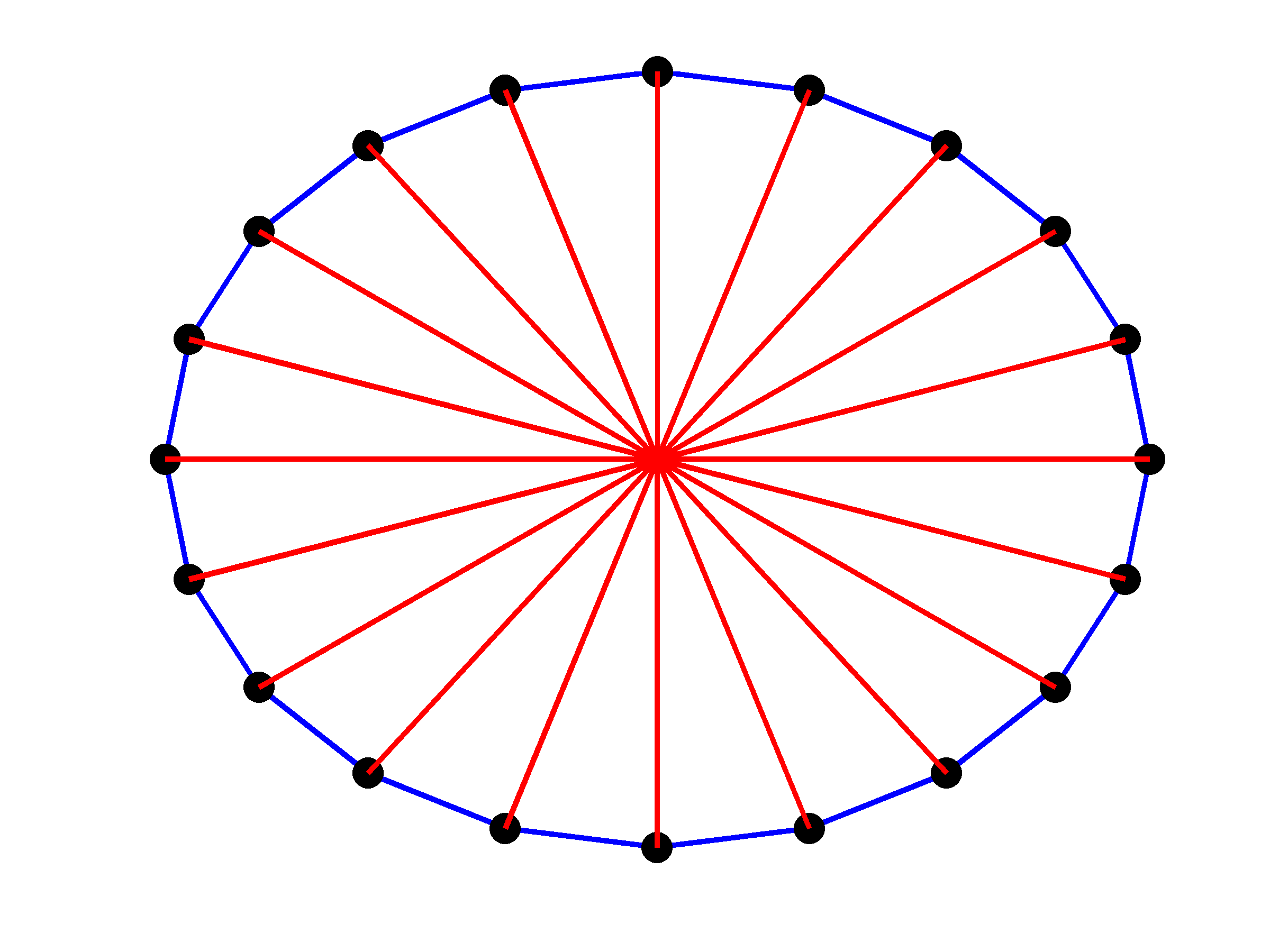}
      }
      \end{tabular}
      \caption{The problems of growing unweighted path (top row) and ring (bottom row) networks. Blue lines identify edges in the plant graph, and red lines identify edges in the controller graph.}
      \label{fig.path}
    \end{figure*}

	\vspace*{-2ex}
\section{Concluding remarks}
	\label{sec.conclusion}
	
We have examined the problem of optimal {topology design} of the corresponding edge weights for undirected consensus networks. Our approach uses convex optimization to balance performance of stochastically-forced networks with the number of edges in the distributed controller. For $\ell_1$-regularized minimum variance optimal control problem, we have derived a Lagrange dual and exploited structure of the optimality conditions for undirected networks to develop customized algorithms that are well-suited for large problems. These are based on the proximal gradient and the proximal Newton methods. The proximal gradient algorithm is a first-order method that updates the controller graph Laplacian via the use of the soft-thresholding operator. {In the proximal Newton method, sequential quadratic approximation of the smooth part of the objective function is employed and the Newton direction} is computed using cyclic coordinate descent over the set of active variables. Examples are provided to demonstrate utility of our algorithms. We have shown that proximal algorithms can solve the problems with millions of edges in the controller graph in several minutes, on a PC. Furthermore, we have specialized our algorithm to the problem of growing connected resistive networks. In this, the plant graph is connected and there are no joint edges between the plant and the controller graphs. We have exploited structure of such networks and demonstrated how additional edges can be systematically added in a computationally efficient manner.


	\vspace*{-2ex}
\section*{Acknowledgments}

We thank J.\ W.\ Nichols for his feedback on earlier versions of this manuscript, T.\ H.\ Summers for useful discussion, and M.\ Sanjabi for his help with C++ implementation.

	\vspace*{-2ex}
	\setstretch{0.89}

\end{document}